\documentclass[11pt]{amsart}
\usepackage{verbatim,amssymb,amsmath,graphicx,hyperref}
\usepackage{xcolor}
  \scrollmode

\newcommand{\EE}{{\mathbb E}}

\newcommand\norm[1]{\lVert#1\rVert}

\theoremstyle{plain}
\newtheorem{theorem}{Theorem}

\newtheorem{lemma}{Lemma}
\newtheorem{cor}{Corollary}

\theoremstyle{definition}

\begin{document}

\title[Lower error bounds for strong approximation of SDEs with discontinuous drift]{Sharp lower error bounds for strong approximation of SDEs with piecewise Lipschitz continuous drift coefficient}

\author[Ellinger]
{Simon Ellinger}
\address{
Faculty of Computer Science and Mathematics\\
University of Passau\\
Innstrasse 33 \\
94032 Passau\\
Germany} \email{simon.ellinger@uni-passau.de}

\begin{abstract}
We study pathwise approximation of strong solutions of scalar stochastic differential equations (SDEs) at a single time in the presence of discontinuities of the drift coefficient. Recently, it has been shown by Müller-Gronbach and Yaroslavtseva (2022) that for all $p \in [1, \infty)$ a transformed Milstein-type scheme reaches an $L^p$-error rate of at least $3 \slash 4$ when the drift coefficient is a piecewise Lipschitz-continuous function with a piecewise Lipschitz-continuous derivative and the diffusion coefficient is constant. It has been proven by Müller-Gronbach and Yaroslavtseva (2023) that this rate $3 \slash 4$ is optimal if one additionally assumes that the drift coefficient is bounded, increasing and has a point of discontinuity. While boundedness and monotonicity of the drift coefficient are crucial for the proof of the matching lower bound from Müller-Gronbach and Yaroslavtseva (2023), we show that both conditions can be dropped. For the proof we apply a transformation technique which was so far only used to obtain upper bounds.
\end{abstract}
\maketitle

\section{INTRODUCTION}
In this paper we study strong approximation of a scalar autonomous stochastic differential equation (SDE)
\begin{flalign} 
\begin{aligned}
 d X_t &= \mu(X_t) \: dt + \sigma(X_t) \: dW_t, \quad t \in [0,1],
\\ X_0 &= x_0,
\end{aligned} \label{eqn:introduction:SDE}
\end{flalign}
where $x_0 \in \mathbb{R}$, $\mu \colon \mathbb{R} \rightarrow \mathbb{R}$ is the drift coefficient and $\sigma \colon \mathbb{R} \rightarrow \mathbb{R}$ is the diffusion coefficient. We derive sharp lower bounds for the $L^p$-error of any method which approximates the solution $X$ at the time $1$ based on finitely many evaluations of the driving Brownian motion $W$ in the case when $\mu$ is piecewise Lipschitz-continuous with piecewise Lipschitz-continuous derivative and $\sigma = 1$.

$L^p$-error rates for strong approximation of SDEs with a discontinuous drift coefficient are only available since about five years. Mainly, the performance of the Euler-scheme and a tamed version hereof was studied up to now, see~\cite{DG18, GLN17, HalidiasKloeden2008,HuGan2022,  LS16,  LS15b, LS18, MGSY22, MGY20, MGY19b, NS19, NSS19, Tag16, Tag2017b, Tag2017a, PS19, Spendier2022}. Here, the most far going result states that the Euler-scheme still achieves the classical $L^2$-error rate of at least $1/2$ if $\mu$ is measurable and bounded and if $\sigma = 1$, see~\cite{DG18}. Furthermore, this rate can be improved if one allows more regularity in the drift coefficient. For instance, in~\cite{NS19} it is shown that the Euler-scheme converges with an $L^2$-error rate of at least $\min\{3 \slash 4, (1 + \kappa) \slash 2\}$ if $\mu$ has $\kappa$-Sobolev-Slobodeckij regularity with $\kappa \in (0,1)$ and $\sigma = 1$. 

In the present paper, we assume that $\sigma = 1$ and $\mu$ satisfies the piecewise smoothness condition
\begin{flalign*}
&(\mu1) \text{\quad There exist a natural number $k \in \mathbb{N}$ as well as $-\infty = \xi_0 < \xi_1 < \dots < \xi_k < \xi_{k+1}=\infty$} &
\\ &\text{\quad\quad\quad such that $\mu$ is Lipschitz-continuous on $(\xi_{i-1}, \xi_i)$ for all $i \in \{1, \dots, k+1\}$,}&
\\
&(\mu2) \text{\quad $\mu$ is differentiable on the interval $(\xi_{i-1}, \xi_i)$ with Lipschitz-continuous derivative}&
\\ &\text{\quad\quad\quad for all $i \in \{1, \dots, k+1\}$.}&
\end{flalign*}
In this case it is known from~\cite{MGY19b} that an $L^p$-error rate $3/4$ for the approximation of $X_1$ using finitely many evaluations of the Brownian motion can be achieved since there exists a sequence of measurable mappings $(g_n)_{n \in \mathbb{N}}$ with $g_n \colon \mathbb{R} \rightarrow \mathbb{R}$ such that for any $p \geq 1$ there exists a constant $C_p > 0$ such that for all $n \in \mathbb{N}$ it holds
\begin{flalign}
\big[\EE [|X_1 - g_n(W_{1 \slash n}, W_{2 \slash n}, \dots, W_1)|^p] \big]^{1/p} \leq \frac{C_p}{n^{3 \slash 4}}. \label{eqn:rateOfTransformedMilstein}
\end{flalign}
We add that \eqref{eqn:rateOfTransformedMilstein} still holds if the condition $\sigma = 1$ is replaced by 
\begin{flalign*}
&(\sigma 1) \text{\quad $\sigma$ is Lipschitz-continuous on $\mathbb{R}$ and it holds $\sigma(\xi_i) \neq 0$ for all $i \in \{1, \dots, k\}$,}&
\\
&(\sigma 2) \text{\quad $\sigma$ is differentiable on the interval $(\xi_{i-1}, \xi_i)$ with Lipschitz-continuous derivative}&
\\&\text{\quad\quad\quad for all $i \in \{1, \dots, k+1\}$,}&
\end{flalign*}
see \cite{MGY19b}.

It is natural to ask whether the rate $3/4$ can be improved in the setting $\sigma = 1$ and $(\mu 1), (\mu 2)$. This was partially answered to the negative in~\cite{MGY21}: If $\sigma = 1$ and $\mu$ satisfies $(\mu 1), (\mu 2)$ as well as 
\begin{flalign*}
&(\mu 3) \text{\quad there exists an $i \in \{1,\dots, k\}$ with $\mu(\xi_{i} +) \neq \mu(\xi_{i} -)$,}&
\\
&(\mu 4) \text{\quad $\mu$ is bounded,}
\\
&(\mu 5) \text{\quad $\mu$ is increasing,}
\end{flalign*}
then there exists a constant $c > 0$ such that for all $n \in \mathbb{N}$ it holds
\begin{flalign}
\inf_{\substack{t_1, \dots, t_n \in [0,1] \\g \colon \mathbb{R}^n \rightarrow \mathbb{R} \: measurable}} \EE |X_1 - g(W_{t_1}, \dots, W_{t_n})| \geq \frac{c}{n^{3 \slash 4}}. \label{eqn:intro:lowerBoundForError}
\end{flalign}
While it is clear that condition $(\mu 3)$ is needed to obtain the lower bound \eqref{eqn:intro:lowerBoundForError}, it was open up to now whether the conditions $(\mu 4)$ and $(\mu 5)$, which are heavily used in the proof of \eqref{eqn:intro:lowerBoundForError} in~\cite{MGY21}, can be dropped. In the present paper we show that this is in fact the case. We substantially modify the technique of the proof in~\cite{MGY21} to obtain the following result:

\begin{theorem} \label{theorem:lowerBound:SDEs:discDrift}
Let $\mu \colon \mathbb{R} \rightarrow \mathbb{R}$ satisfy $(\mu 1), (\mu 2)$ and $(\mu 3)$. Let $x_0 \in \mathbb{R}$ and let $X: [0,1] \times \Omega \rightarrow \mathbb{R}$ be a strong solution of the SDE 
\begin{flalign}
dX_t = \mu(X_t) \: dt + dW_t \label{eqn:basicSDE}
\end{flalign}
on the time interval $[0,1]$ with initial value $x_0$ and driving Brownian motion $W$. Then there exists a constant $c > 0$ such that for all $n \in \mathbb{N}$,
\begin{flalign*}
\inf_{\substack{t_1, \dots, t_n \in [0,1] \\g \colon \mathbb{R}^n \rightarrow \mathbb{R} \: measurable}} \EE |X_1 - g(W_{t_1}, \dots, W_{t_n})| \geq \frac{c}{n^{3 \slash 4}}.
\end{flalign*} 
\end{theorem}

Above theorem also holds for non-constant diffusion coefficients if the drift coefficient is bounded. As can be seen in the following corollary, it suffices that the diffusion coefficient is elliptic and in the space $C^3_b(\mathbb{R})$ of bounded, three times differentiable functions $\mathbb{R} \rightarrow \mathbb{R}$ with bounded first, second and third derivative.

\begin{cor}\label{cor:GeneralLowerBound}
	Let $\mu \colon \mathbb{R} \rightarrow \mathbb{R}$ be bounded and satisfy $(\mu 1), (\mu 2), (\mu 3)$ and let $\sigma \in C^3_b(\mathbb{R})$ with $\inf_{x \in \mathbb{R}} |\sigma(x)| > 0$. Let $x_0 \in \mathbb{R}$ and let $X: [0,1] \times \Omega \rightarrow \mathbb{R}$ be a strong solution of the SDE 
	\begin{flalign*}
		dX_t = \mu(X_t) \: dt + \sigma(X_t) \:dW_t 
	\end{flalign*}
	on the time interval $[0,1]$ with initial value $x_0$ and driving Brownian motion $W$. Then there exists a constant $c > 0$ such that for all $n \in \mathbb{N}$,
	\begin{flalign*}
		\inf_{\substack{t_1, \dots, t_n \in [0,1] \\g \colon \mathbb{R}^n \rightarrow \mathbb{R} \: measurable}} \EE |X_1 - g(W_{t_1}, \dots, W_{t_n})| \geq \frac{c}{n^{3 \slash 4}}.
	\end{flalign*} 
\end{cor}

We note that this lower bound does not hold if one considers methods that are based on sequential evaluations of $W$. In fact, in~\cite{Y21} a Milstein scheme with adaptive step-size control is constructed that achieves an $L^p$-error rate of at least $1$ in terms of the average number of evaluations of $W$ if the conditions $(\mu 1), (\mu 2), (\sigma 1)$ and $(\sigma 2)$ are satisfied.
\\

\section{PROOF OF THE MAIN RESULT}
We briefly present the structure of this section. First, we outline the proof of Theorem \ref{theorem:lowerBound:SDEs:discDrift} in Subsection \ref{subsection:proofIdea}. In Subsection \ref{subsection:transformationAndFurtherTools} we give an overview over preliminary work and we will introduce a transformation which is used in the proof of Theorem \ref{theorem:lowerBound:SDEs:discDrift}. Then we carry out the proof of Theorem \ref{theorem:lowerBound:SDEs:discDrift} in Subsection \ref{subsection:proofOfMainResult}. Finally, we prove Corollary \ref{cor:GeneralLowerBound} in Subsection \ref{subsection:proofOfCorollary}.
\subsection{Idea of the proof}\label{subsection:proofIdea}
The proof of Theorem \ref{theorem:lowerBound:SDEs:discDrift} relies on a lower bound for the investigated $L^2$-error of approximation which is independent of the specific function $g \colon \mathbb{R}^n \rightarrow \mathbb{R}$. This lower bound will be given by the $L^2$-distance of our final time point $X_1$ and the final time point $\widetilde{X}_1$ of another solution of the SDE \eqref{eqn:basicSDE} with driving Brownian motion $\widetilde{W}$. Thereby, the Brownian motion $\widetilde{W}$ is constructed with the original Brownian motion $W$ such that $W$ and $\widetilde{W}$ coincide at the grid points $t_1, \dots, t_n$ and such that $W$ and $\widetilde{W}$ are independent given $W_{t_1}, \dots, W_{t_n}$. To be precise, we use the piecewise linear interpolation $\overline{W}$ of the Brownian motion $W$ in order to define the process $B = W - \overline{W}$. The process $B$ consists of Brownian bridges on each of the intervals $[t_0, t_1], \dots, [t_{n-1}, t_n]$ which are independent of each other. The idea now is to introduce another Brownian bridge process $\widetilde{B}$ with $\mathbb{P}^B = \mathbb{P}^{\widetilde{B}}$ which is independent of $W, B$ and to set
\begin{flalign*}
\widetilde{W} = \overline{W} + \widetilde{B}.
\end{flalign*}
Then we consider the two solutions $X, \widetilde{X}$ of the SDE \eqref{eqn:basicSDE} which are of the form
\begin{flalign*}
X_t = x_0 + \int_0^t \mu(X_s) \: ds + W_t,  \quad \widetilde{X}_t = x_0 + \int_0^t \mu(\widetilde{X}_s) \: ds + \widetilde{W}_t, \quad \quad t \in [0,1].
\end{flalign*}
As already mentioned above, we use the process $\widetilde{X}$ to obtain a lower bound for the approximation error and as we will see, one gets for every measurable function $g \colon \mathbb{R}^n \rightarrow \mathbb{R}$
\begin{flalign*}
\big[ \EE |X_1 - g(W_{t_1}, \dots, W_{t_n})|^2 \big]^{1 \slash 2} \geq \frac{1}{2} \big[ \EE |X_1 - \widetilde{X}_1|^2 \big]^{1 \slash 2}.
\end{flalign*}
For the estimation of the right expression in the above inequality we investigate the distance of the transformations of our solutions. Such transformations were used several times in the literature to obtain upper bounds, see for instance~\cite{LS16},~\cite{LS15b} and~\cite{MGY19b}. Similar to~\cite{MGY19b}, we use a transform $G_{\mu}$ which is Lipschitz-continuous and has a Lipschitz-continuous inverse $G_{\mu}^{-1}$. Since $G_{\mu}$ is Lipschitz-continuous with some Lipschitz-constant $L_{G_{\mu}}$, one sees that
\begin{flalign*}
\EE \big| X_1 - \widetilde{X}_1|^2 \geq (L_{G_{\mu}})^{-2} \EE |G_\mu(X_1) - G_\mu(\widetilde{X}_1)|^2.
\end{flalign*}
Letting
\begin{flalign*}
\widetilde{\mu} := (G_\mu^\prime \cdot \mu + \frac{1}{2}G_{\mu}^{\prime \prime}) \circ G_{\mu}^{-1} \quad \text{and} \quad \widetilde{\sigma} := G_{\mu}^\prime \circ G_{\mu}^{-1},
\end{flalign*}
one can show that $\widetilde{\mu}$ and $\widetilde{\sigma}$ are Lipschitz-continuous and that $G_\mu \circ X$ is a strong solution of the SDE 
\begin{flalign}
dY_t = \widetilde{\mu}(Y_t) \: dt + \widetilde{\sigma}(Y_t) \: dW_t. \label{eqn:basicTildeSDE}
\end{flalign}
The benefit of this transformed process $G_\mu \circ X$ is that the coefficients of the SDE are Lipschitz-continuous which allows us to apply known stability results with respect to the initial values. More precisely, we will set for $i \in \{1, \dots, n\}$ with $t_i \geq \frac{1}{2}$
\begin{flalign*}
\Delta_i := \EE|G_{\mu}(X_{t_i}) - G_{\mu}(\widetilde{X}_{t_i})|^2.
\end{flalign*}
Assume that $t_{i} > 1/2$. Rewriting above definition immediately yields
\begin{flalign*}
\Delta_i &= \EE \big[|G_{\mu}(X_{t_{i-1}}) - G_{\mu}(\widetilde{X}_{t_{i-1}}) + ((G_{\mu}(X_{t_i}) - G_{\mu}(X_{t_{i-1}})) - (G_{\mu}(\widetilde{X}_{t_i}) - G_{\mu}(\widetilde{X}_{t_{i-1}})))|^2 \big] 
\\&= \Delta_{i-1} + 2m_i + d_i
\end{flalign*}
with
\begin{flalign*}
m_i := \EE\big[(G_{\mu}(X_{t_{i-1}}) - G_{\mu}(\widetilde{X}_{t_{i-1}})) \cdot ((G_{\mu}(X_{t_i}) - G_{\mu}(X_{t_{i-1}})) - (G_{\mu}(\widetilde{X}_{t_i}) - G_{\mu}(\widetilde{X}_{t_{i-1}})))\big]
\end{flalign*}
and
\begin{flalign*} 
d_i := \EE|(G_{\mu}(X_{t_i}) - G_{\mu}(X_{t_{i-1}})) - (G_{\mu}(\widetilde{X}_{t_i}) - G_{\mu}(\widetilde{X}_{t_{i-1}}))|^2.
\end{flalign*}
Then we can bound the mixed terms in a suitable way and we will derive a constant $\hat{C}_1 > 0$ with
\begin{flalign*}
|m_i| \leq \frac{\hat{C}_1}{n} \Delta_{i-1}.
\end{flalign*}
Above estimate for the mixed terms $m_i$ can be obtained fast but we have to put more effort into finding a lower bound for the part with the $d_i$-terms. Applying again another transformation to the processes $(G_\mu(X_{s + t_{i-1}}) - G_\mu(X_{t_{i-1}}))_{s \in [0, t_i - t_{i-1}]}$ and $(G_\mu(\widetilde{X}_{ s+ t_{i-1}}) - G_\mu(\widetilde{X}_{t_{i-1}}))_{s \in [0, t_i  - t_{i-1}]}$ conditioned on $(X_{t_{i-1}}, \widetilde{X}_{t_{i-1}})$ will yield the existence of constants $\hat{c}_1, \hat{C}_2 > 0$ such that
\begin{flalign*}
d_i \geq \hat{c}_1(t_i - t_{i-1})^{5 \slash 2} - \frac{\hat{C}_2}{n} \Delta_{i-1} - o(n^{- 5 \slash 2}).
\end{flalign*}
Combining the estimates for the mixed terms $m_i$ and for the diagonal terms $d_i$ shows the existence of constants $c_1, C_1> 0$ such that for all $j \in \{1, \dots, n\}$ with $t_j \geq\frac{1}{2}$ it holds
\begin{flalign*}
\Delta_n &\geq (1- \frac{C_1}{n})^{n-j} \Delta_j + c_1 \sum_{i=j+1}^n(1- \frac{C_1}{n})^{n-i} (t_i - t_{i-1})^{5 \slash 2} - (n-j) \cdot o(n^{- 5 \slash 2})
\\& \geq  c_1 (1- \frac{C_1}{n})^{n} \sum_{i=j+1}^n (t_i - t_{i-1})^{5 \slash 2} -   o(n^{- 3 \slash 2}).
\end{flalign*}
Taking $j = r$ with $t_r = \frac{1}{2}$ yields the claim for the $L^2$-error instead of the $L^1$-error if one notices that $\lim_{n \rightarrow \infty} (1- \frac{C_1}{n})^n = e^{-C_1}$ and that it holds by the H\"older-inequality
\begin{flalign*}
\sum_{i=r+1}^n(t_i - t_{i-1})^{5 \slash 2} \geq \frac{1}{2^{5 \slash 2}n^{3 \slash 2}}.
\end{flalign*}
The lower bound for the $L^1$-error then follows by a standard technique.

\subsection{The transformation and further tools}\label{subsection:transformationAndFurtherTools}
In this section we introduce the already mentioned transformation $G_\mu$ which will allow us to deal with an SDE with Lipschitz-continuous coefficients instead of the original SDE \eqref{eqn:basicSDE}. Such transformations were already applied in~\cite{LS16},~\cite{LS15b} and~\cite{MGY19b} to obtain upper bounds for the Euler-scheme or for a Milstein-type scheme. Subsequently, we present in Section \ref{subsubsection:gernalResults_mu_sde} basic properties of $(\mu 1)$-functions, of SDEs with additive noise and of SDEs in the classical setting. Afterwards, the transformation $G_\mu$ will be introduced in Section \ref{subsubsection:transformation} and we restate some results which are proven in~\cite{MGY21} and which we will use later. Moreover, we give in Section \ref{subsubsection:couplingOfNoise} a short introduction to the lower bound technique presented therein.

\subsubsection{General results for $(\mu 1)$-functions and SDEs}\label{subsubsection:gernalResults_mu_sde}
Firstly, we summarize some basic properties of a $(\mu 1)$-function. In particular, we see in the following lemma that a $(\mu 1)$-function can be written as the sum of a Lipschitz-continuous function and a step function which will be used to drop the monotonicity condition on $\mu$.
\begin{lemma}\label{lemma:functionalsWelldefinedMuOfLinearGrowth}
	Let $\mu \colon \mathbb{R} \rightarrow \mathbb{R}$ be a function satisfying $(\mu1)$. Then it holds:
	\begin{itemize}  
		\item[(i)] The function $\mu$ satisfies the linear growth property and for $i \in \{1, \dots, k\}$ the limits $\mu(\xi_i-) = \lim_{x \uparrow \xi_i} \mu(x)$ and $\mu(\xi_i+) = \lim_{x \downarrow \xi_i} \mu(x)$ exist. Moreover, with $D_i := \{(x,y) \in \mathbb{R}^2 : (x - \xi_i)(y -\xi_i)\leq 0\}$ for $i \in \{1, \dots, k\}$ there exists a constant $C > 0$ such that for all $x,y \in \mathbb{R}$ it holds
		\begin{flalign*}
			|\mu(x) - \mu(y)| \leq C(|x-y| + \sum_{i=1}^k 1_{D_i}(x,y)).
		\end{flalign*}
		\item[(ii)] There exist a Lipschitz-continuous function $\mu_{cont} \colon \mathbb{R} \rightarrow \mathbb{R}$ and real numbers $\alpha_1, \dots, \alpha_k,$ $\gamma_1, \dots, \gamma_k \in \mathbb{R}$ such that
			\begin{flalign*}
				\mu = \mu_{cont} + \sum_{i=1}^k \alpha_i 1_{[\xi_i, \infty)} + \sum_{i=1}^k \gamma_i 1_{\{\xi_i\}}.
		\end{flalign*}
	\end{itemize}
\end{lemma}

\begin{proof}
		The first statement coincides with Lemma 1 in~\cite{MGY21}. The second claim follows by induction over the number of discontinuities of $\mu$. For simplicity, we only show the claim for $k=1$. Set 
		\begin{flalign*}
			\mu_{cont} = 1_{(-\infty, \xi_1)} \cdot \mu + 1_{\{\xi_1\}} \cdot \mu(\xi_1 -) + 1_{(\xi_1, \infty)} \cdot(\mu(x) - \mu(\xi_1 +) + \mu(\xi_1 -)),
		\end{flalign*}
		$\alpha_1 = (\mu(\xi_1+) - \mu(\xi_1 -))$ and $\gamma_1 = (\mu(\xi_1) - \mu(\xi_1+))$. Since $\mu_{cont}$ is continuous and Lipschitz-continuous on $(-\infty, \xi_1)$ and on $(\xi_1, \infty)$, $\mu_{cont}$ is Lipschitz-continuous and the claim is proven in this case.
\end{proof}

Subsequently, we present properties of SDEs with additive noise and a drift satisfying $(\mu 1)$. In the next lemma we show that the probability that a solution of the SDE \eqref{eqn:basicSDE} is in an interval can be bounded up to some constants by the length of the interval.
\begin{lemma} \label{lemma:boundProbability:SolutionIsInInterval}
	Let $\mu \colon \mathbb{R} \rightarrow \mathbb{R}$ be a measurable function satisfying the linear growth property, let $(\Omega, \mathcal{F}, \mathbb{P})$ be a complete probability space, let $W \colon [0,1] \times \Omega \rightarrow \mathbb{R}$ be a Brownian motion, let $x_0 \in \mathbb{R}$ and let $X \colon [0,1] \times \Omega \rightarrow \mathbb{R}$ be a strong solution of the SDE \eqref{eqn:basicSDE} on the time interval $[0,1]$ with driving Brownian motion $W$ and initial value $x_0$. Let $\tau \in (0,1]$ and $M > 0$. Then there exist constants $c,C > 0$ such that for all $t \in [\tau, 1]$ and for all $a,b \in [-M,M]$ with $a \leq b$ it holds
	\begin{flalign*}
		c (b-a) \leq \mathbb{P}(X_t \in [a,b]) \leq C(b - a).
	\end{flalign*}
\end{lemma}

\begin{proof}
	First of all, we know with Theorem 2.1 in~\cite{FP10} that for $t \in (0,1]$ the distribution $\mathbb{P}^{X_t}$ has a  Lebesgue density $p_t$ since $\mu$ satisfies the linear growth property. Therefore, it suffices to show the inequalities
	\begin{flalign}
		0 < \inf_{t \in [\tau, 1]} \inf_{x \in [-M, M]} p_t(x) \leq \sup_{t \in [\tau, 1]} \sup_{x \in [-M, M]} p_t(x) < \infty. \label{eqn:lemma:boundProbability:SolutionIsInInterval:infSupBound}
	\end{flalign}
	Let $K_\mu > 0$ be such that $|\mu(x)| \leq K_\mu (1 + |x|)$ for all $x \in \mathbb{R}$. Now it holds 
	\begin{flalign*}
		|\mu(x)| \leq K_\mu(1 + |x_0| + |x - x_0|) \quad \quad \text{for all $x \in \mathbb{R}$.}
	\end{flalign*}
	Hence, applying Theorem 3.1 and Theorem 3.2 in~\cite{QZ04} yields the existence of constants $c_i,C_i > 0$ for $i\in \{1,\dots, 6\}$ such that for all $t \in (0, 1]$ and $x \in \mathbb{R}$ we have
	\begin{flalign*}
		p_t(x) &\geq \frac{c_1}{\sqrt{e^{c_2t} - 1}} \cdot  \exp(- c_3  \frac{|x - x_0|^2}{e^{c_2t } - 1} - c_3(1 + |x_0|) (1 + |x -x_0|) - c_4(1 + |x_0|)^2t)
		\intertext{as well as}
		p_t(x) &\leq \frac{C_1}{\sqrt{1- e^{-C_2t}}} + \frac{C_3 (1+ |x_0|)}{\sqrt{1- e^{-C_2 t}}} \cdot (1+ |x - x_0|) \cdot \exp(- \frac{C_4 |x - x_0|^2}{(1- e^{-C_2 t}) e^{C_5 t}} + C_6 (1+|x_0|)^2).
	\end{flalign*}
	Thus, it holds \eqref{eqn:lemma:boundProbability:SolutionIsInInterval:infSupBound} which finishes the proof.
\end{proof}

With the above lemmas we can now investigate how much time a solution of the SDE \eqref{eqn:basicSDE} spends on a different side of some real point than its approximation.
	
\begin{lemma}\label{lemma:lowerBound:dropRecursionIntegralOfSDE:boundJumpPositios}
	Let $\mu \colon \mathbb{R} \rightarrow \mathbb{R}$ be a function satisfying $(\mu 1)$. Let $x_0, \xi \in \mathbb{R}$ and $X \colon [0,1] \times \Omega \rightarrow \mathbb{R}$ be a strong solution of the SDE \eqref{eqn:basicSDE} on the time interval $[0,1]$ with initial value $x_0$ and driving Brownian motion $W$. Then there exists a constant $C > 0$ such that for all $\frac{1}{2} \leq s \leq t\leq 1$ it holds
	\begin{flalign*}
		\EE |\int_s^t 1_{\{(X_{s} + W_u - W_{s} - \xi)(X_u - \xi) \leq 0\}} \: du|^2 \leq  C (t-s)^{5/2 + 1/6}.
	\end{flalign*}
\end{lemma}

\begin{proof}
	Let $\frac{1}{2} \leq s \leq t\leq 1$. Using the definition of a strong solution we see that for $u \in [s,t]$ it holds 
	\begin{flalign*}
		X_s + W_u - W_s = X_u - \int_{s}^{u} \mu(X_v) \: dv.
	\end{flalign*}
	Therewith, we obtain
	\begin{flalign*}
		\EE |\int_{s}^{t} 1_{\{(X_s + W_u - W_s - \xi)(X_u - \xi) \leq 0\}} \: du|^2 
		&\leq (t - s) \EE \int_{s}^{t} 1_{\{( X_u - \int_{s}^u \mu(X_v) \: dv - \xi)(X_u - \xi) \leq 0\}} \: du
		\\& \leq(t - s) \int_{s}^{t} \mathbb{P}(| X_u - \xi| \leq |\int_{s}^{u} \mu(X_v) \: dv|) \: du.
	\end{flalign*}
	Now fix $\varepsilon \in (0, \infty)$, let $u \in [s,t]$ and let us proceed similar to the proof of Lemma 14 in~\cite{MGY21}. It holds
	\begin{flalign*}
		\mathbb{P}(| X_u - \xi| \leq |\int_{s}^{u} \mu(X_v) \: dv|) \leq \mathbb{P}(| X_u - \xi| \leq (t-s)^{1/2 + \varepsilon}) + \mathbb{P}((t-s)^{ 1/2 + \varepsilon} \leq |\int_{s}^{u} \mu(X_v) \: dv|)
	\end{flalign*}
	and due to Lemma \ref{lemma:functionalsWelldefinedMuOfLinearGrowth} and Lemma \ref{lemma:boundProbability:SolutionIsInInterval} there exists a constant $C_1 > 0$ such that 
		\begin{flalign*}
			\mathbb{P}(| X_u - \xi| \leq  (t-s)^{1/2 + \varepsilon}) \leq C_1 (t-s)^{1/2 + \varepsilon}.
		\end{flalign*}
		Later we show the existence of some constant $C_2 > 0$ which is independent of $s,t,u$ and which satisfies
	\begin{flalign}
		\mathbb{P}( (t-s)^{1/2 + \varepsilon}\leq |\int_{s}^{u} \mu(X_v) \: dv|) \leq C_2 (t-s)^{1 - 2 \varepsilon}. \label{eqn:lowerBound:dropRecursionIntegralOfSDE:boundJumpPositios:boundIntegralPart}
	\end{flalign}
	Then the claim follows with the choice $\varepsilon = \frac{1}{6}$.
	\\Let us start with the derivation of the above $C_2$. We see with an application of Markov's inequality that it holds 
	\begin{flalign*}
		\mathbb{P}((t-s)^{1/2 + \varepsilon} \leq |\int_{s}^{u} \mu(X_v) \: dv|) \leq (t-s)^{- 1 - 2\varepsilon} \EE|\int_{s}^{u} \mu(X_v) \: dv|^2 \leq (t-s)^{- 2\varepsilon} \int_{s}^{t} \EE|\mu(X_v)|^2 \: dv.
	\end{flalign*}
	Since $\mu$ satisfies the linear growth property according to Lemma \ref{lemma:functionalsWelldefinedMuOfLinearGrowth} there exists a constant $C_2 > 0$ such that $\EE [\sup_{v \in [0,1]} |\mu(X_v)|^2] \leq  C_2$. Hence, we obtain the validity of
	\begin{flalign*}
		\mathbb{P}((t-s)^{1/2 + \varepsilon} \leq |\int_{s}^{u} \mu(X_v) \: dv|) \leq   C_2(t-s)^{1 - 2\varepsilon}.
	\end{flalign*}
	Altogether, the existence of $C_2 > 0$ in \eqref{eqn:lowerBound:dropRecursionIntegralOfSDE:boundJumpPositios:boundIntegralPart} is shown which finishes the proof.
\end{proof}

Besides the above estimation for the approximation of a strong solution, it will be helpful to have a functional relation between a solution of the SDE \eqref{eqn:basicSDE} and the initial value together with the driving Brownian motion. This functional relation is presented in the next lemma which one obtains similar to Lemma 9 in~\cite{MGY21}. 

\begin{lemma} \label{lemma:existenceOfFunctionF}
	Let $\mu$ be a function satisfying $(\mu 1)$. Then for every $T \in (0, \infty)$ there exists a Borel-measurable function
	\begin{flalign*}
		F \colon \mathbb{R} \times C([0,T], \mathbb{R}) \rightarrow C([0,T], \mathbb{R})
	\end{flalign*}
	such that for every complete probability space $(\Omega, \mathcal{F}, \mathbb{P})$, every Brownian motion $W \colon [0,T] \times \Omega \rightarrow \mathbb{R}$ and every random variable $\eta \colon \Omega \rightarrow \mathbb{R}$ such that $W, \eta$ are independent it holds:
	\begin{itemize}
		\item[(i)] if $X \colon [0,T] \times \Omega \rightarrow \mathbb{R}$ is a strong solution of the SDE \eqref{eqn:basicSDE} on the time interval $[0,T]$ with driving Brownian motion $W$ and initial value $\eta$, then $\mathbb{P}$-almost surely it holds $X = F(\eta, W)$,
		
		\item[(ii)] $F(\eta, W)$ is a strong solution of the SDE \eqref{eqn:basicSDE} on the time interval $[0,T]$ with driving Brownian motion $W$ and initial value $\eta$.
	\end{itemize}
\end{lemma}

We will later transform a solution of the SDE \eqref{eqn:basicSDE} to a solution of an SDE which satisfies classical assumptions. In the classical setting we can then use that the distance of two solutions can be controlled by the distance of the initial values as can be seen in the next lemma.

\begin{lemma} \label{lemma:L1DistnaceOfStrongSolutions:differenceOfInitialValue}
	Let $\widetilde{\mu}, \widetilde{\sigma} \colon \mathbb{R} \rightarrow \mathbb{R}$ be Lipschitz-continuous functions. Then there exists a constant $C > 0$ such that for all complete probability spaces $(\Omega, \mathcal{F}, \mathbb{P})$, all $\tau \in (0,1]$, all Brownian motions $V \colon [0, \tau] \times \Omega \rightarrow \mathbb{R}$, all $y, \tilde{y} \in \mathbb{R}$ and all strong solutions $Y^y, Y^{\tilde{y}} \colon \Omega \times [0,\tau] \rightarrow \mathbb{R}$ of the SDE
	\begin{flalign*}
		dY_t = \widetilde{\mu}(Y_t) \: dt + \widetilde{\sigma}(Y_t) \: dV_t
	\end{flalign*}
	with initial values $Y^y_0 = y$ and $Y^{\tilde{y}}_0 = \tilde{y}$ it holds
	\begin{flalign*}
		\EE |Y_s^y - Y_s^{\tilde{y}}| \leq C |y - \tilde{y}|, \quad s \in [0, \tau].
	\end{flalign*}
\end{lemma}

\begin{proof}
	Let $(\Omega, \mathcal{F}, \mathbb{P})$ be a complete probability space. The claim follows with the basic inequality $\EE|X| \leq [\EE |X|^2]^{1 \slash 2}$ for a random variable $X$ and with standard arguments as in the proof of Theorem 9.2.4 in~\cite{RevuzYor1999}.
\end{proof}

\subsubsection{On the transformation}\label{subsubsection:transformation}
Let us continue with the introduction of the transformation $G_\mu$ for a $(\mu 1)$-function $\mu$. The next lemma is a straightforward generalization of \mbox{Lemma 1} in~\cite{MGY19b} as well as Lemma 3 in~\cite{MGSY22} and its proof.

\begin{lemma} \label{lemma:basicPropertiesOfG}
Let $\mu$ be a function satisfying $(\mu1)$. Then there exists a function $G_{\mu}$ which is a strictly monotonically increasing Lipschitz-continuous function with the following properties:
\begin{itemize}
\item[(i)] $G_{\mu}$ is differentiable and has a bounded Lipschitz-continuous derivative $G_{\mu}^\prime$ with 
		 \\ $\inf_{x\in \mathbb{R}} G_{\mu}^\prime(x)> 0$,
\item[(ii)] $G_{\mu}$ has a Lipschitz-continuous inverse $G_{\mu}^{-1} \colon \mathbb{R} \rightarrow \mathbb{R}$,
\item[(iii)] for all $i \in \{1, \dots, k+1\}$ the function $G_{\mu}^\prime$ restricted to $(\xi_{i-1}, \xi_i)$ is differentiable and has a bounded and Lipschitz-continuous derivative $G_{\mu}^{\prime \prime}$,
\item[(iv)] for all $i \in \{1, \dots, k\}$ the limits $G_\mu^{\prime \prime}(\xi_i -) = \lim_{x \uparrow \xi_i} G_\mu^{\prime \prime}(x)$ and $G_\mu^{\prime \prime}(\xi_i +) = \lim_{x \downarrow \xi_i} G_\mu^{\prime \prime}(x)$ exist and it holds
\begin{flalign*}
G_\mu^{\prime \prime}(\xi_i -) = - (\mu(\xi_i -) - \mu(\xi_i +)), \quad G_\mu^{\prime \prime}(\xi_i +) =  \mu(\xi_i -) - \mu(\xi_i +),
\end{flalign*}
\item[(v)] $G_{\mu}^{-1}$ is differentiable and has a Lipschitz-continuous derivative $(G_{\mu}^{-1})^\prime$ and it holds
\begin{flalign*}
(G_{\mu}^{-1})^\prime(x) = \frac{1}{G_{\mu}^\prime(G_{\mu}^{-1}(x))}, \quad x \in \mathbb{R},
\end{flalign*}
\item[(vi)] for all $x \in \mathbb{R} \setminus \{G_\mu(\xi_i) : i \in \{1, \dots, k\}\}$ the derivative $(G_{\mu}^{-1})^\prime$ is differentiable in $x$ and its derivative is given by 
\begin{flalign*}
(G_{\mu}^{-1})^{\prime \prime}(x) = -\frac{G_{\mu}^{\prime \prime}(G_{\mu}^{-1}(x))}{(G_{\mu}^\prime(G_{\mu}^{-1}(x)))^3},  \quad x \in \mathbb{R}.
\end{flalign*} 
\end{itemize}
\end{lemma}

Subsequently, the function $G_\mu$ is the function of Lemma \ref{lemma:basicPropertiesOfG} for a $(\mu 1)$-function $\mu$. Moreover, for a $(\mu 1)$-function $\mu$ we extend in consideration of Lemma \ref{lemma:basicPropertiesOfG} the second derivatives
\\$G_{\mu}^{\prime \prime} \colon \cup_{i=1}^{k+1} (\xi_{i-1}, \xi_i) \rightarrow \mathbb{R}$ and $(G_{\mu}^{-1})^{\prime \prime} \colon \mathbb{R} \setminus \{G(\xi_i) : i \in \{1, \dots, k\}\} \rightarrow \mathbb{R}$ to the whole line as in~\cite{MGY19b} via 
\begin{flalign*}
&G_{\mu}^{\prime \prime}(\xi_i) := (\mu(\xi_i -) - \mu(\xi_i +)) + 2(\mu(\xi_i +) - \mu(\xi_i)), \quad i \in \{1, \dots, k\},
\\&(G_{\mu}^{-1})^{\prime \prime}(G_\mu(\xi_i)) := -\frac{G_{\mu}^{\prime \prime}(\xi_i)}{(G_{\mu}^\prime(\xi_i))^3}, \quad i \in \{1, \dots, k\}.
\end{flalign*}

As already mentioned, it is of key importance that the mapping $G_\mu$ transforms a solution of the SDE \eqref{eqn:basicSDE} into a solution of an SDE with Lipschitz-continuous coefficients. The corresponding statement and the exact form of the transformed coefficients can be seen in the following lemma. The next lemma follows from Lemma 2 in~\cite{MGY19b} and the proof of Lemma 9 in~\cite{MGY21}.

\begin{lemma} \label{lemma:GTransformsStrongSolutions}
Let $\mu$ be a function satisfying $(\mu 1)$ and let
\begin{flalign*}
\widetilde{\mu} := (G_\mu^\prime \cdot \mu + \frac{1}{2}G_{\mu}^{\prime \prime}) \circ G_{\mu}^{-1} \quad \text{and} \quad \widetilde{\sigma} := G_{\mu}^\prime \circ G_{\mu}^{-1}.
\end{flalign*}
Then $\widetilde{\mu}$ and $\widetilde{\sigma}$ are Lipschitz-continuous and we have for every $T \in (0, \infty)$, every complete probability space $(\Omega, \mathcal{F}, \mathbb{P})$, every Brownian motion $W \colon [0,T]\times \Omega \rightarrow \mathbb{R}$ and every random variable $\eta: \Omega \rightarrow \mathbb{R}$ such that $W,\eta$ are independent:
\begin{itemize}
\item[(i)] if $X \colon [0,T] \times \Omega \rightarrow \mathbb{R}$ is a strong solution of the SDE
\begin{flalign*}
dX_t = \widetilde{\mu}(X_t) \: dt + \widetilde{\sigma}(X_t) \: dW_t 
\end{flalign*}
on the time interval $[0,T]$ with driving Brownian motion $W$ and initial value $\eta$, then $G_{\mu}^{-1} \circ X$ is a strong solution of the SDE \eqref{eqn:basicSDE} on the time interval $[0,T]$ with driving Brownian motion $W$ and initial value $G_{\mu}^{-1}(\eta)$,

\item[(ii)] if $X \colon [0,T] \times \Omega \rightarrow \mathbb{R}$ is a strong solution of the SDE \eqref{eqn:basicSDE} on the time interval $[0,T]$ with driving Brownian motion $W$ and initial value $\eta$, then $G_{\mu} \circ X$ is a strong solution of the SDE \eqref{eqn:basicTildeSDE} on the time interval $[0,T]$ with driving Brownian motion $W$ and initial value $G_{\mu}(\eta)$. 
\end{itemize}
\end{lemma}

Afterwards, we will always use the expressions $\widetilde{\mu}$ and $\widetilde{\sigma}$ for the transformed coefficients from Lemma \ref{lemma:GTransformsStrongSolutions} when we deal with a function $\mu$ satisfying $(\mu 1)$. In order to come back from the transformed SDE with Lipschitz-continuous coefficients to the original SDE with $\sigma = 1$, we will later apply the next lemma. The lemma follows from Lemma \ref{lemma:basicPropertiesOfG} with elementary computations.

\begin{lemma} \label{lemma:fromTildeToMu}
Let $\mu \colon \mathbb{R} \rightarrow \mathbb{R}$ be a function satisfying $(\mu 1)$ and let $\widetilde{\mu}, \widetilde{\sigma}$ be given as in Lemma \ref{lemma:GTransformsStrongSolutions}. Then for all $x \in \mathbb{R}$ it holds
\begin{flalign*}
(G_{\mu}^{-1})^\prime(x) \cdot \widetilde{\mu}(x) + \frac{1}{2}(G_{\mu}^{-1})^{\prime \prime}(x) \cdot
\widetilde{\sigma}^2(x) = \mu(G_{\mu}^{-1}(x)) \quad \text{and} \quad (G_{\mu}^{-1})^\prime(x) \cdot \widetilde{\sigma}(x) = 1.
\end{flalign*}
\end{lemma}

\subsubsection{On the coupling of noise}\label{subsubsection:couplingOfNoise}
The transformation $G_\mu$ comes into play together with the lower bound technique from~\cite{MGY21} which we will briefly introduce now. First of all, it holds for $m \in \mathbb{N}$ with $m \geq 2$
\begin{flalign*}
	\inf_{\substack{t_1, \dots, t_{m-1} \in [0,1] \\g \colon \mathbb{R}^{m-1} \rightarrow \mathbb{R} \: measurable}} &\EE |X_1 - g(W_{t_1}, \dots, W_{t_{m-1}})|
	\\&\geq \inf_{\substack{t_1, \dots, t_m \in [0,1] \\g \colon \mathbb{R}^m \rightarrow \mathbb{R} \: measurable}} \EE |X_1 - g(W_{t_1}, \dots, W_{t_m})|
	\\&\geq \inf_{\substack{t_1, \dots, t_m \in [0,1] \\g \colon \mathbb{R}^{2m} \rightarrow \mathbb{R} \: measurable}} \EE |X_1 - g(W_{t_1}, \dots, W_{t_m}, W_{1/m}, \dots, W_{(m-1)/m}, W_1)|.
\end{flalign*}

Therefore, it suffices for a proof of Theorem \ref{theorem:lowerBound:SDEs:discDrift} to compute with an $n \in 2\mathbb{N}$ such that $n \geq 16$ and to show the existence of some constant $c > 0$ such that $c$ is independent of $n$ and such that for all
\begin{flalign}
0 < t_1 < t_2 < \dots < t_n = 1 \label{eqn:thmLowerBoundOccupationTime:tiIn0To1}
\end{flalign}
with
\begin{flalign}
\{ \frac{2}{n}, \frac{4}{n}, \dots , \frac{n-2}{n}, 1\} \subset \{t_1, \dots, t_n\} \label{eqn:thmLowerBoundOccupationTime:2OverNInti}
\end{flalign}
it holds
\begin{flalign*}
\inf_{g \colon \mathbb{R}^n \rightarrow \mathbb{R} \: measurable} \EE|X_1 - g(W_{t_1}, \dots, W_{t_n})|\geq  \frac{c}{n^{3 \slash 4}}.
\end{flalign*}

So let us fix $t_1, \dots, t_n \in [0,1]$ with \eqref{eqn:thmLowerBoundOccupationTime:tiIn0To1} as well as \eqref{eqn:thmLowerBoundOccupationTime:2OverNInti} and put $t_0 := 0$. Every constant that we will derive afterwards will be independent from $n$ and from the specific choice of $t_1,\dots, t_n$ if not stated otherwise. We will use the linear interpolation $\overline{W}$ of the Brownian motion $W$ with sample positions at $t_1, \dots, t_n$ which is given for $t \in [t_{i-1}, t_i]$ with $i \in \{1, \dots, n\}$ by
\begin{flalign*}
\overline{W}_t = \frac{t - t_{i-1}}{t_i - t_{i-1}} W_{t_i} + \frac{t_i - t}{t_i - t_{i-1}} W_{t_{i-1}}.
\end{flalign*}
Obviously, $t_1, \dots, t_n$ are coincident points of the processes $W$ and $\overline{W}$. Between the sample points  $t_1, \dots, t_n$ the process $W - \overline{W}$ behaves like a Brownian bridge. To be more formal, let us define
\begin{flalign*}
B := W - \overline{W}.
\end{flalign*}
Now the process $(B_t)_{t \in [t_{i-1}, t_i]}$ is a Brownian bridge on $[t_{i-1}, t_i]$ for all $i \in \{1, \dots, n\}$. Moreover, the processes $(B_t)_{t \in [t_0, t_1]}, (B_t)_{t \in [t_1, t_2]}, \dots, (B_t)_{t \in [t_{n-1},t_n]}, \overline{W}$ are independent. With this in mind, we replace the Brownian bridges with new independent ones. Therefore, let $(\widetilde{B}_t)_{t \in [t_{i-1}, t_i]}$ be a Brownian bridge on $[t_{i-1}, t_i]$ for $i \in \{1, \dots, n\}$ such that
\\\centerline{$(\widetilde{B}_t)_{t \in [t_0, t_1]}, (\widetilde{B}_t)_{t \in [t_1, t_2]}, \dots,$ $ (\widetilde{B}_t)_{t \in [t_{n-1},t_n]}, W$}
are independent. Let us put $\widetilde{B} = (\widetilde{B}_t)_{t \in [0,1]}$ and define because of $W = \overline{W} + B$ a new Brownian motion $\widetilde{W}$ by
\begin{flalign*}
\widetilde{W} := \overline{W} + \widetilde{B}.
\end{flalign*}

The following lemma shows us why the new Brownian motion $\widetilde{W}$ is of interest for the study of lower bounds. The next statement corresponds to Lemma 11 in~\cite{MGY21}.

\begin{lemma}\label{lemma:lowerBound:infBoundedByTildeDistance}
Let $\mu \colon \mathbb{R} \rightarrow \mathbb{R}$ be a function satisfying $(\mu 1)$. Let $x_0 \in \mathbb{R}$ and $X,\widetilde{X} \colon [0,1] \times \Omega \rightarrow \mathbb{R}$ be strong solutions of the SDE \eqref{eqn:basicSDE} on the time interval $[0,1]$ with initial value $x_0$ and driving Brownian motion $W$ and $\widetilde{W}$, respectively. Then for every measurable function $g \colon \mathbb{R}^n \rightarrow \mathbb{R}$ and for every $p \in [1, \infty)$ it holds
\begin{flalign*}
\big( \EE [|X_1 - g(W_{t_1}, \dots, W_{t_n})|^p] \big)^{1 \slash p} \geq \frac{1}{2} (\EE[|X_1 - \widetilde{X}_1|^p])^{1 \slash p}.
\end{flalign*}
\end{lemma}

We will apply the measurable mapping $F$ from Lemma \ref{lemma:existenceOfFunctionF} in the context of strong solutions of the SDE \eqref{eqn:basicSDE} with driving Brownian motion $W$ and $\widetilde{W}$, respectively. This will be done using the next statement which also says that the future behavior of the Brownian motion is independent of the strong solution of the SDE at the current point of time. The following lemma can be shown with the same arguments used in the proof of Lemma 13 in~\cite{MGY21}.

\begin{lemma} \label{lemma:representationOfXWithF}
Let $\mu \colon \mathbb{R} \rightarrow \mathbb{R}$ be a function satisfying $(\mu 1)$. Let $x_0 \in \mathbb{R}$ and $X,\widetilde{X} \colon [0,1] \times \Omega \rightarrow \mathbb{R}$ be strong solutions of the SDE \eqref{eqn:basicSDE} on the time interval $[0,1]$ with initial value $x_0$ and driving Brownian motion $W$ and $\widetilde{W}$, respectively. Let $i \in \{1, \dots, n\}$, $V := (V_t = W_{t_{i-1} + t} - W_{t_{i-1}})_{t \in [0, t_i - t_{i-1}]}$ and $\widetilde{V} := (\widetilde{V}_t = \widetilde{W}_{t_{i-1} + t} - \widetilde{W}_{t_{i-1}})_{t \in [0, t_i - t_{i-1}]}$. Then the processes $(X_{t_{i-1}}, \widetilde{X}_{t_{i-1}})$ and $(V, \widetilde{V})$ are independent. Moreover,  with $F \colon \mathbb{R} \times C([0, t_i - t_{i-1}], \mathbb{R}) \rightarrow C([0, t_i - t_{i-1}], \mathbb{R})$ as introduced in Lemma \ref{lemma:existenceOfFunctionF} it holds $\mathbb{P}$-almost surely
\begin{flalign*}
(X_{t_{i-1} + t})_{t \in [0, t_i - t_{i-1}]} &= F(X_{t_{i-1}}, V),
\\
(\widetilde{X}_{t_{i-1} + t})_{t \in [0, t_i - t_{i-1}]} &= F(\widetilde{X}_{t_{i-1}}, \widetilde{V}).
\end{flalign*}
\end{lemma}

For the estimation of the diagonal terms $d_i$ we use the following bound of the maximum distance of two solutions of the SDE \eqref{eqn:basicSDE} with driving Brownian motion $W$ and $\widetilde{W}$, respectively, at the time points $t_1, \dots, t_n$.

\begin{lemma} \label{lemma:boundLPDistanceOfSolutionsAtInterpolation}
Let $\mu \colon \mathbb{R} \rightarrow \mathbb{R}$ be a function satisfying $(\mu 1)$ and $(\mu 2)$. Let $x_0 \in \mathbb{R}$ and $X, \widetilde{X}\colon [0,1] \times \Omega \rightarrow \mathbb{R}$ be strong solutions of the SDE \eqref{eqn:basicSDE} with initial value $x_0$ and driving Brownian motion $W$ and $\widetilde{W}$, respectively. Then for every $p \in [1, \infty)$ there exists a constant $C>0$ such that
\begin{flalign*}
\max_{i \in \{0, \dots, n\}} \big[ \EE |X_{t_i} - \widetilde{X}_{t_i}|^p \big]^{1 \slash p} \leq \frac{C}{n^{3 \slash 4}}.
\end{flalign*}
\end{lemma}

\begin{proof}
The statement can be shown analogously to Lemma 12 in~\cite{MGY21} if one uses instead of the boundedness of $\mu$ the fact that there exists a constant $C > 0$ with
\begin{flalign}
\EE |X_{t_i} - X_{\underline{t_i}} - \widetilde{X}_{t_i} + \widetilde{X}_{\underline{t_i}}|^p \leq C n^{-p}, \quad i \in \{0,\dots, n\}, \label{eqn:lemma:boundLPDistanceOfSolutionsAtInterpolation:boundErrorPart}
\end{flalign}
where we set
\begin{flalign*}
\underline{t_i} := \max\{ \tau \in \{2j \slash n: j = 0, \dots, n \slash 2\} : t_i \geq \tau\}, \quad i \in \{0, \dots, n\}.
\end{flalign*}
Let $i \in \{0, \dots, n\}$. We will only show \eqref{eqn:lemma:boundLPDistanceOfSolutionsAtInterpolation:boundErrorPart}. It holds
\begin{flalign*}
\EE |X_{t_i} - X_{\underline{t_i}} - \widetilde{X}_{t_i} + \widetilde{X}_{\underline{t_i}}|^p &= \EE |\int_{\underline{t_i}}^{t_i} (\mu(X_s) - \mu(\widetilde{X}_s)) \: ds|^p \leq (2 / n)^{p-1}  \EE \int_{\underline{t_i}}^{t_i} |\mu(X_s) - \mu(\widetilde{X}_s)|^p \: ds.
\end{flalign*}
Since $\mu$ satisfies the linear growth property according to Lemma \ref{lemma:functionalsWelldefinedMuOfLinearGrowth} and since we have $\mathbb{P}^{X_s} = \mathbb{P}^{\widetilde{X}_s}$ for $s \in [0,1]$, there exist constants $C_1, C_2, C_3 > 0$ which are independent of $i$ such that
\begin{flalign*}
\EE |X_{t_i} - X_{\underline{t_i}} - \widetilde{X}_{t_i} + \widetilde{X}_{\underline{t_i}}|^p &\leq C_ 1 \cdot (2/n)^{p-1}  \int_{\underline{t_i}}^{t_i} \EE [(1 + |X_s|)^p] \: ds
\\& \leq C_2 \cdot (2/n)^p \cdot \EE[1 + \sup_{s \in [0,1]}|X_s|^p]
\\& \leq C_3 \cdot (1/n)^p.
\end{flalign*}
Hence, we have shown \eqref{eqn:lemma:boundLPDistanceOfSolutionsAtInterpolation:boundErrorPart}.
\end{proof}


\subsection{Proof of Theorem \ref{theorem:lowerBound:SDEs:discDrift}} \label{subsection:proofOfMainResult}
In this section we start with preparatory work for the proof of Theorem \ref{theorem:lowerBound:SDEs:discDrift} and afterwards we prove the theorem. Our strategy is to derive a statement which holds in the classical setting with Lipschitz-continuous coefficients. Thereafter, we can bound for $i \in \{1, \dots, n\}$ the mixed terms
\begin{flalign*}
m_i = \EE\big[(G_{\mu}(X_{t_{i-1}}) - G_{\mu}(\widetilde{X}_{t_{i-1}})) \cdot ((G_{\mu}(X_{t_i}) - G_{\mu}(X_{t_{i-1}})) - (G_{\mu}(\widetilde{X}_{t_i}) - G_{\mu}(\widetilde{X}_{t_{i-1}})))\big]
\end{flalign*}
of a function $\mu$ satisfying $(\mu 1)$ in a suitable fashion. Afterwards, we will successively derive lower bounds for the diagonal terms
\begin{flalign*} 
d_i = \EE|(G_{\mu}(X_{t_i}) - G_{\mu}(X_{t_{i-1}})) - (G_{\mu}(\widetilde{X}_{t_i}) - G_{\mu}(\widetilde{X}_{t_{i-1}}))|^2.
\end{flalign*}

So, let us start with the derivation of the bounds for the mixed terms. We will later condition on $(X_{t_{i-1}}, \widetilde{X}_{t_{i-1}}) = (z, \tilde{z}) \in \mathbb{R}^2$ in order to obtain with Lemma \ref{lemma:representationOfXWithF} two strong solutions of the SDE with drift coefficient $\widetilde{\mu}$ and diffusion coefficient $\widetilde{\sigma}$ and with driving Brownian motion $V$ as well as deterministic initial values $G_\mu(z)$ and $G_\mu(\tilde{z})$, respectively.

\begin{lemma} \label{lemma:boundForMixedTerms}
Let $\mu \colon \mathbb{R} \rightarrow \mathbb{R}$ be a function satisfying $(\mu 1)$. Let $x_0 \in \mathbb{R}$ and $X,\widetilde{X} \colon [0,1] \times \Omega \rightarrow \mathbb{R}$ be strong solutions of the SDE \eqref{eqn:basicSDE} on the time interval $[0,1]$ with initial value $x_0$ and with driving Brownian motion $W$ and $\widetilde{W}$, respectively. Then there exists a constant $C > 0$ such that for all $i \in \{1, \dots, n\}$ it holds
\begin{flalign*}
|\EE &[(G_{\mu}(X_{t_{i-1}}) - G_{\mu}(\widetilde{X}_{t_{i-1}})) \cdot  (G_{\mu}(X_{t_i}) - G_{\mu}(X_{t_{i-1}}) - (G_{\mu}(\widetilde{X}_{t_i}) - G_{\mu}(\widetilde{X}_{t_{i-1}})))]| 
\\&\leq \frac{C}{n} \EE|G_{\mu}(X_{t_{i-1}}) - G_{\mu}(\widetilde{X}_{t_{i-1}})|^2.
\end{flalign*}
\end{lemma}

\begin{proof}
Let $i \in \{1, \dots, n\}$. In this proof we use $\EE_{z, \tilde{z}}[Y] := \EE[Y | (X_{t_{i-1}}, \widetilde{X}_{t_{i-1}}) = (z, \tilde{z})]$ for an $L^1$-random variable $Y$ and for $z, \tilde{z} \in \mathbb{R}$. Let $F,V$ and $\widetilde{V}$ be given as in Lemma \ref{lemma:representationOfXWithF}. Then the same lemma yields that $(X_{t_{i-1}}, \widetilde{X}_{t_{i-1}})$ and $(V, \widetilde{V})$ are independent and that it holds $\mathbb{P}$-almost surely
\begin{flalign}
\begin{aligned}
(X_{t_{i-1} + t})_{t \in [0, t_i - t_{i-1}]} &= F(X_{t_{i-1}}, V),
\\
(\widetilde{X}_{t_{i-1} + t})_{t \in [0, t_i - t_{i-1}]} &= F(\widetilde{X}_{t_{i-1}}, \widetilde{V}). 
\end{aligned} \label{eqn:representationOfXWithF}
\end{flalign}
Consequently, it holds for $\mathbb{P}^{(X_{t_{i-1}}, \widetilde{X}_{t_{i-1}})}$-almost all $(z, \tilde{z}) \in \mathbb{R}^2$ 
\begin{flalign*}
&\EE_{z, \tilde{z}} [(G_{\mu}(X_{t_{i-1}}) - G_{\mu}(\widetilde{X}_{t_{i-1}})) \cdot (G_{\mu}(X_{t_i}) - G_{\mu}(X_{t_{i-1}}) - (G_{\mu}(\widetilde{X}_{t_i}) - G_{\mu}(\widetilde{X}_{t_{i-1}})))]
\\&\quad= (G_{\mu}(z) - G_{\mu}(\tilde{z})) \cdot \EE[(G_{\mu}(F(z, V)(t_i - t_{i-1})) - G_{\mu}(z)) - (G_{\mu}(F(\tilde{z}, \widetilde{V})(t_i - t_{i-1})) - G_{\mu}(\tilde{z}))]
\\&\quad= (G_{\mu}(z) - G_{\mu}(\tilde{z})) \cdot \EE[(G_{\mu}(F(z, V)(t_i - t_{i-1})) - G_{\mu}(z)) - (G_{\mu}(F(\tilde{z}, V)(t_i - t_{i-1})) - G_{\mu}(\tilde{z}))]
\end{flalign*}
According to Lemma \ref{lemma:existenceOfFunctionF}, now $F(z, V)$ and $F(\tilde{z}, V)$ are strong solutions of the SDE \eqref{eqn:basicSDE} on the time interval $[0,t_i - t_{i-1}]$ with driving Brownian motion $V$ and initial value $z$ and $\tilde{z}$, respectively. Therefore, by Lemma \ref{lemma:GTransformsStrongSolutions} the processes $Y^{G_{\mu}(z)} := G_{\mu}(F(z, V))$ and $Y^{G_{\mu}(\tilde{z})}:=G_{\mu}(F(\tilde{z}, V))$ are strong solutions of the SDE \eqref{eqn:basicTildeSDE} on the time interval $[0,t_i - t_{i-1}]$ with driving Brownian motion $V$ and initial value $G_{\mu}(z)$ and $G_{\mu}(\tilde{z})$, respectively. Therewith, it hods 
\begin{flalign*}
\EE&[(G_{\mu}(F(z, V)(t_i - t_{i-1})) - G_{\mu}(z)) - (G_{\mu}(F(\tilde{z}, V)(t_i - t_{i-1})) - G_{\mu}(\tilde{z}))] 
\\&= \int_0^{t_i - t_{i-1}} \EE [\widetilde{\mu}(Y^{G_{\mu}(z)}_s) - \widetilde{\mu}(Y^{G_{\mu}(\tilde{z})}_s)] \: ds + \EE \int_0^{t_i - t_{i-1}}(\widetilde{\sigma}(Y^{G_{\mu}(z)}_s) - \widetilde{\sigma}(Y^{G_{\mu}(\tilde{z})}_s)) \: dV_s.
\end{flalign*}
Let $L_{\widetilde{\mu}} > 0$ denote the Lipschitz-constant of $\widetilde{\mu}$. Since $(\int_0^t \widetilde{\sigma}(Y_u^{G(z)}) \: dV_u)_{t \in [0, t_i - t_{i-1}]}$ and 
\\$(\int_0^t \widetilde{\sigma}(Y_u^{G(\tilde{z})}) \: dV_u)_{t \in [0, t_i - t_{i-1}]}$ are martingales, it holds for $\mathbb{P}^{(X_{t_{i-1}}, \widetilde{X}_{t_{i-1}})}$-almost all $(z, \tilde{z}) \in \mathbb{R}^2$ 
\begin{flalign*}
|\EE_{z, \tilde{z}} &[(G_{\mu}(X_{t_{i-1}}) - G_{\mu}(\widetilde{X}_{t_{i-1}})) \cdot (G_{\mu}(X_{t_i}) - G_{\mu}(X_{t_{i-1}}) - (G_{\mu}(\widetilde{X}_{t_i}) - G_{\mu}(\widetilde{X}_{t_{i-1}})))]|
\\&\leq |G_{\mu}(z) - G_{\mu}(\tilde{z})| \cdot L_{\widetilde{\mu}} \int_0^{t_i - t_{i-1}} \EE |Y^{G_{\mu}(z)}_s - Y^{G_{\mu}(\tilde{z})}_s| \: ds.
\end{flalign*}
Hence, an application of Lemma \ref{lemma:L1DistnaceOfStrongSolutions:differenceOfInitialValue} finishes the proof.
\end{proof}

Now we estimate the diagonal terms $d_i$. Thereby, we will step by step derive lower bounds that depend on the expression $\frac{1}{n} \Delta_i$ and some terms which are in $o(n^{-5 \slash 2})$. For this, we will again condition on $(X_{t_{i-1}}, \widetilde{X}_{t_{i-1}}) = (z, \tilde{z})$ and then we will transform the occurring processes $Y^{G_\mu(z)} - G_\mu(z)$ and $\widetilde{Y}^{G_\mu(\tilde{z})} - G_\mu(\tilde{z})$ which satisfy
\begin{flalign*}
	\mathbb{P}&^{(Y^{G_\mu(z)} - G_\mu(z), \widetilde{Y}^{G_\mu(\tilde{z})} - G_\mu(\tilde{z}))} 
	\\&= \mathbb{P}^{((G_\mu(X_{t_{i-1} + s}) - G_\mu(X_{t_{i-1}}), G_\mu(\widetilde{X}_{t_{i-1} + s}) -G_\mu(\widetilde{X}_{t_{i-1}})))_{s \in [0, t_i - t_{i-1}]} | (X_{t_{i-1}}, \widetilde{X}_{t_{i-1}}) = (z, \tilde{z})}.
\end{flalign*}
This will yield a new expression which can be bounded in a suitable way.

\begin{lemma}\label{lemma:lowerBound:increments}
Let $\mu \colon \mathbb{R} \rightarrow \mathbb{R}$ be a function satisfying $(\mu 1)$. Let $x_0 \in \mathbb{R}$ and $X,\widetilde{X} \colon [0,1] \times \Omega \rightarrow \mathbb{R}$ be strong solutions of the SDE \eqref{eqn:basicSDE} on the time interval $[0,1]$ with initial value $x_0$ and driving Brownian motion $W$ and $\widetilde{W}$, respectively. Then there exist constants $c, C  > 0$ such that for all $i \in \{1, \dots, n\}$ it holds
\begin{flalign}
\EE& |G_{\mu}(X_{t_i}) - G_{\mu}(X_{t_{i-1}}) - (G_{\mu}(\widetilde{X}_{t_i}) - G_{\mu}(\widetilde{X}_{t_{i-1}}))|^2 \notag
\\&\geq c \EE |\int_{t_{i-1}}^{t_i} \mu(X_s) - \mu(G_{\mu}^{-1}(G_{\mu}(\widetilde{X}_s) - G_{\mu}(\widetilde{X}_{t_{i-1}}) + G_{\mu}(X_{t_{i-1}}))) \: ds|^2  \label{eqn:lowerBound:increments:statement}
\\& \quad- \frac{C}{n}\EE|G_{\mu}(X_{t_{i-1}}) - G_{\mu}(\widetilde{X}_{t_{i-1}})|^2. \notag
\end{flalign}
\end{lemma}

\begin{proof}
Let $i \in \{1, \dots, n\}$ and $F,V$ as well as $\widetilde{V}$ be given as in Lemma \ref{lemma:representationOfXWithF}. Arguing analogously to the proof of Lemma \ref{lemma:boundForMixedTerms} shows that for $(z, \tilde{z}) \in \mathbb{R}^2$ the process $Y^{G_{\mu}(z)} := G_{\mu}(F(z, V))$ is a strong solution of the SDE \eqref{eqn:basicTildeSDE} on the time interval $[0,t_i - t_{i-1}]$ with driving Brownian motion $V$ and initial value $G_{\mu}(z)$ and that the process $\widetilde{Y}^{G_{\mu}(\tilde{z})}:=G_{\mu}(F(\tilde{z}, \widetilde{V}))$ is a strong solution of the SDE \eqref{eqn:basicTildeSDE} on the time interval $[0,t_i - t_{i-1}]$ with driving Brownian motion $\widetilde{V}$ and initial value $G_{\mu}(\tilde{z})$. Since $(X_{t_{i-1}}, \widetilde{X}_{t_{i-1}})$ and $(V, \widetilde{V})$ are independent according to Lemma \ref{lemma:representationOfXWithF} and because of \eqref{eqn:representationOfXWithF}, it holds for $\mathbb{P}^{(X_{t_{i-1}}, \widetilde{X}_{t_{i-1}})}$-almost all $(z, \tilde{z}) \in \mathbb{R}^2$ 
\begin{flalign}
\begin{aligned}
\EE &[|G_{\mu}(X_{t_i}) - G_{\mu}(X_{t_{i-1}}) - (G_{\mu}(\widetilde{X}_{t_i}) - G_{\mu}(\widetilde{X}_{t_{i-1}}))|^2 | (X_{t_{i-1}}, \widetilde{X}_{t_{i-1}}) = (z, \tilde{z})] 
\\&= \EE|(Y^{G_{\mu}(z)}_{t_i - t_{i-1}} - G_{\mu}(z)) - (\widetilde{Y}^{G_{\mu}(\tilde{z})}_{t_i - t_{i-1}} - G_{\mu}(\tilde{z}))|^2. 
\end{aligned}\label{eqn:lowerBound:increments:transformedIncrements}
\end{flalign}
The main step is now to apply the It\^{o}-formula with a suitable function. The function which we will use for the It\^{o}-formula is the inverse of
\begin{flalign*}
H_{\mu, z}\colon \mathbb{R} \rightarrow \mathbb{R}, \quad y \mapsto G_{\mu}(y) - G_{\mu}(z).
\end{flalign*}
Since $G_{\mu}$ is a bijection, we immediately see that the inverse of $H_{\mu, z}$ exists and is given by 
\begin{flalign*}
H_{\mu, z}^{-1}(y) = G_{\mu}^{-1}(y + G_{\mu}(z)), \quad y \in \mathbb{R}.
\end{flalign*}
Our aim is now to apply the It\^{o}-formula with $H_{\mu, z}^{-1}$ on $(Y_s^{G_{\mu}(z)} - G_{\mu}(z))_{s \in [0, t_i - t_{i-1}]}$ and $(\widetilde{Y}_s^{G_{\mu}(\tilde{z})} - G_{\mu}(\tilde{z}))_{s \in [0, t_i - t_{i-1}]}$ in \eqref{eqn:lowerBound:increments:transformedIncrements}. Therefore, let $L_{G_{\mu}^{-1}}$ be the Lipschitz-constant of $G_{\mu}^{-1}$. Then it holds
\begin{flalign}
\begin{aligned}
 \EE& |(Y^{G_{\mu}(z)}_{t_i - t_{i-1}} - G_{\mu}(z)) - (\widetilde{Y}^{G_{\mu}(\tilde{z})}_{t_i - t_{i-1}} - G_{\mu}(\tilde{z}))|^2 
\\&=  \EE|H_{\mu, z}(H_{\mu, z}^{-1}(Y^{G_{\mu}(z)}_{t_i - t_{i-1}} - G_{\mu}(z))) - H_{\mu, z}(H_{\mu, z}^{-1}(\widetilde{Y}^{G_{\mu}(\tilde{z})}_{t_i - t_{i-1}} - G_{\mu}(\tilde{z})))|^2 
\\&= \EE|G_{\mu}(H_{\mu, z}^{-1}(Y^{G_{\mu}(z)}_{t_i - t_{i-1}} - G_{\mu}(z))) - G_{\mu}(H_{\mu, z}^{-1}(\widetilde{Y}^{G_{\mu}(\tilde{z})}_{t_i - t_{i-1}} - G_{\mu}(\tilde{z})))|^2 
\\& \geq (L_{G_{\mu}^{-1}})^{-2} \EE|(H_{\mu, z}^{-1}(Y^{G_{\mu}(z)}_{t_i - t_{i-1}} - G_{\mu}(z))) - (H_{\mu, z}^{-1}(\widetilde{Y}^{G_{\mu}(\tilde{z})}_{t_i - t_{i-1}} - G_{\mu}(\tilde{z})))|^2. 
\end{aligned}\label{eqn:lowerBound:increment:transformed}
\end{flalign}
Now the It\^{o}-formula of Problem 3.7.3 in~\cite{ks91} can be applied since $H_{\mu,z}^{-1}$ is differentiable due to Lemma \ref{lemma:basicPropertiesOfG} and has a Lipschitz-continuous derivative $(H_{\mu,z}^{-1})^\prime$ which is given by
\begin{flalign}
(H_{\mu,z}^{-1})^\prime(y) = (G_{\mu}^{-1})^\prime (y + G_{\mu}(z)), \quad y \in \mathbb{R}.  \label{eqn:firstDerivative:H}
\end{flalign}
In particular, $(H_{\mu,z}^{-1})^\prime$ is absolutely continuous and with Lemma \ref{lemma:basicPropertiesOfG} it holds for $\lambda$-almost all $y \in \mathbb{R}$
\begin{flalign}
(H_{\mu,z}^{-1})^{\prime \prime} (y) = (G_{\mu}^{-1})^{\prime \prime} (y + G_{\mu}(z)). \label{eqn:secondDerivative:H}
\end{flalign}
Applying the It\^{o}-formula in consideration of \eqref{eqn:firstDerivative:H} as well as \eqref{eqn:secondDerivative:H} yields that it holds $\mathbb{P}$-almost surely
\begin{flalign*}
H_{\mu, z}^{-1}&(Y^{G_{\mu}(z)}_{t_i - t_{i-1}} - G_{\mu}(z))
\\ &= H_{\mu, z}^{-1}(0) + \int_0^{t_i - t_{i-1}} (H_{\mu, z}^{-1})^\prime(Y^{G_{\mu}(z)}_s - G_{\mu}(z)) \cdot \widetilde{\mu}(Y^{G_{\mu}(z)}_s)  \: ds 
\\&\quad+ \frac{1}{2}\int_0^{t_i - t_{i-1}}(H_{\mu,z}^{-1})^{\prime \prime}(Y^{G_{\mu}(z)}_s - G_{\mu}(z)) \cdot
\widetilde{\sigma}^2(Y^{G_{\mu}(z)}_s) \: ds 
\\& \quad+ \int_0^{t_i - t_{i-1}} (H_{\mu, z}^{-1})^\prime(Y^{G_{\mu}(z)}_s - G_{\mu}(z)) \cdot \widetilde{\sigma}(Y^{G_{\mu}(z)}_s) \: dV_s
\\&= H_{\mu, z}^{-1}(0) + \int_0^{t_i - t_{i-1}} (G_{\mu}^{-1})^\prime(Y^{G_{\mu}(z)}_s) \cdot \widetilde{\mu}(Y^{G_{\mu}(z)}_s) + \frac{1}{2}(G_{\mu}^{-1})^{\prime \prime}(Y^{G_{\mu}(z)}_s) \cdot
\widetilde{\sigma}^2(Y^{G_{\mu}(z)}_s) \: ds 
\\& \quad+ \int_0^{t_i - t_{i-1}} (G_{\mu}^{-1})^\prime(Y^{G_{\mu}(z)}_s) \cdot \widetilde{\sigma}(Y^{G_{\mu}(z)}_s) \: dV_s. 
\end{flalign*}
Lemma \ref{lemma:fromTildeToMu} now shows that it holds $\mathbb{P}$-almost surely
\begin{flalign}
H_{\mu, z}^{-1}&(Y^{G_{\mu}(z)}_{t_i - t_{i-1}} - G_{\mu}(z)) = H_{\mu, z}^{-1}(0) + \int_0^{t_i - t_{i-1}} \mu(G_{\mu}^{-1}(Y^{G_{\mu}(z)}_s)) \: ds + W_{t_i} - W_{t_{i-1}}. \label{eqn:nonTildeProcess:transformed}
\end{flalign}
Analogously, we obtain with the It\^{o}-formula applied to $(\widetilde{Y}_s^{G_{\mu}(\tilde{z})} - G_{\mu}(\tilde{z}))_{s \in [0, t_i - t_{i-1}]}$ and $H_{\mu, z}^{-1}$ that it holds $\mathbb{P}$-almost surely
\begin{flalign*}
H_{\mu, z}^{-1}(\widetilde{Y}^{G_{\mu}(\tilde{z})}_{t_i - t_{i-1}} - G_{\mu}(\tilde{z}))
&=  H_{\mu, z}^{-1}(0) + \int_0^{t_i - t_{i-1}} (G_{\mu}^{-1})^\prime(\widetilde{Y}_s^{G_{\mu}(\tilde{z})} - G_{\mu}(\tilde{z}) + G_{\mu}(z)) \cdot \widetilde{\mu}(\widetilde{Y}^{G_{\mu}(\tilde{z})}_s) \: ds
\\& \quad + \frac{1}{2}\int_0^{t_i - t_{i-1}}(G_{\mu}^{-1})^{\prime \prime}(\widetilde{Y}^{G_{\mu}(\tilde{z})}_s - G_{\mu}(\tilde{z}) + G_{\mu}(z)) \cdot \widetilde{\sigma}^2(\widetilde{Y}^{G_{\mu}(\tilde{z})}_s) \: ds
\\& \quad + \int_0^{t_i - t_{i-1}} (G_{\mu}^{-1})^\prime(\widetilde{Y}^{G_{\mu}(\tilde{z})}_s - G_{\mu}(\tilde{z}) + G_{\mu}(z)) \cdot \widetilde{\sigma}(\widetilde{Y}^{G_{\mu}(\tilde{z})}_s) \: d\widetilde{V}_s.
\end{flalign*}
For simplicity, let us set for the shifted process $\widetilde{U}^{z,\tilde{z}} := \widetilde{Y}^{G_{\mu}(\tilde{z})} - G_{\mu}(\tilde{z}) + G_{\mu}(z)$. Another application of Lemma \ref{lemma:fromTildeToMu} yields due to $W_{t_i} = \widetilde{W}_{t_i}$ for $i \in \{1, \dots, n\}$ that it holds $\mathbb{P}$-almost surely
\begin{equation}\label{eqn:tildeProcess:transformed}
\begin{aligned}
H_{\mu, z}^{-1}(\widetilde{Y}^{G_{\mu}(\tilde{z})}_{t_i - t_{i-1}} - G_{\mu}(\tilde{z})) 
&= H_{\mu, z}^{-1}(0) + \int_0^{t_i - t_{i-1}} \mu(G_{\mu}^{-1}(\widetilde{Y}^{G_{\mu}(\tilde{z})}_s - G_{\mu}(\tilde{z}) + G_{\mu}(z))) \: ds + W_{t_i} - W_{t_{i-1}} 
\\& \quad+ A_{z, \tilde{z}} + B_{z, \tilde{z}} + C_{z, \tilde{z}} 
\end{aligned}
\end{equation}
with 
\begin{flalign*}
A_{z, \tilde{z}} &:= \int_0^{t_i - t_{i-1}} (G_{\mu}^{-1})^\prime(\widetilde{U}^{z,\tilde{z}}_s) \cdot (\widetilde{\mu}(\widetilde{Y}^{G_{\mu}(\tilde{z})}_s) - \widetilde{\mu}(\widetilde{U}^{z,\tilde{z}}_s)) \: ds,
\\
B_{z,\tilde{z}} &:= \int_0^{t_i - t_{i-1}}\frac{1}{2}(G_{\mu}^{-1})^{\prime \prime}(\widetilde{U}^{z,\tilde{z}}_s) \cdot(\widetilde{\sigma}^2(\widetilde{Y}^{G_{\mu}(\tilde{z})}_s) - \widetilde{\sigma}^2(\widetilde{U}^{z,\tilde{z}}_s))\: ds,
\\
C_{z, \tilde{z}} &:= \int_0^{t_i - t_{i-1}} (G_{\mu}^{-1})^\prime(\widetilde{U}^{z,\tilde{z}}_s) \cdot(\widetilde{\sigma}(\widetilde{Y}^{G_{\mu}(\tilde{z})}_s) - \widetilde{\sigma}(\widetilde{U}^{z,\tilde{z}}_s)) \: d\widetilde{V}_s.
\end{flalign*}

Using $\frac{1}{2}(a+b)^2 - b^2 \leq a^2$ for $a,b \in \mathbb{R}$ and plugging \eqref{eqn:nonTildeProcess:transformed} and \eqref{eqn:tildeProcess:transformed} into \eqref{eqn:lowerBound:increment:transformed} shows that

\begin{flalign}
\begin{aligned}
 \EE& |(Y^{G_{\mu}(z)}_{t_i - t_{i-1}} - G_{\mu}(z)) - (\widetilde{Y}^{G_{\mu}(\tilde{z})}_{t_i - t_{i-1}} - G_{\mu}(\tilde{z}))|^2 
\\& \geq (L_{G_{\mu}^{-1}})^{-2}[\frac{1}{8} \EE|\int_0^{t_i - t_{i-1}} \mu(G_{\mu}^{-1}(Y^{G_{\mu}(z)}_s)) - \mu(G_{\mu}^{-1}(\widetilde{Y}^{G_{\mu}(\tilde{z})}_s - G_{\mu}(\tilde{z}) + G_{\mu}(z)))\: ds|^2 
\\&  \quad - \EE |A_{z,\tilde{z}}|^2 - \EE|B_{z,\tilde{z}}|^2 - \EE|C_{z,\tilde{z}}|^2 ]. 
\end{aligned}\label{eqn:increment:lowerBound:conditioned2}
\end{flalign}

We will later derive a constant $C_1 > 0$ which is independent of $z, \tilde{z}$ and which satisfies
\begin{flalign}
\max\{ \EE |A_{z,\tilde{z}}|^2, \EE|B_{z,\tilde{z}}|^2, \EE|C_{z,\tilde{z}}|^2\} \leq \frac{C_1}{n}\cdot |G_{\mu}(z) - G_{\mu}(\tilde{z})|^2. \label{eqn:boundForABC}
\end{flalign}

Now it holds due to the definition of $Y^{G_{\mu}(X_{t_{i-1}})}$ and due to \eqref{eqn:representationOfXWithF} $\mathbb{P}$-almost surely
\begin{flalign*}
Y^{G_{\mu}(X_{t_{i-1}})} &=  G_{\mu}(F(X_{t_{i-1}}, V)) = (G_{\mu}(X_{t_{i-1} + t}))_{t \in [0, t_i - t_{i-1}]},
\\\widetilde{Y}^{G_{\mu}(\widetilde{X}_{t_{i-1}})} &=  G_{\mu}(F(\widetilde{X}_{t_{i-1}}, \widetilde{V})) =  (G_{\mu}(\widetilde{X}_{t_{i-1} + t}))_{t \in [0, t_i - t_{i-1}]}.
\end{flalign*}

Recalling that $(X_{t_{i-1}}, \widetilde{X}_{t_{i-1}})$ and $(V, \widetilde{V})$ are independent, the claim follows with \eqref{eqn:lowerBound:increments:transformedIncrements}, \eqref{eqn:increment:lowerBound:conditioned2} and \eqref{eqn:boundForABC}.

It remains to show \eqref{eqn:boundForABC}. First of all, we note that with Lemma \ref{lemma:basicPropertiesOfG} and with the definition of $\widetilde{\sigma}$ it holds
\begin{flalign*}
\norm{\widetilde{\sigma}}_{\infty}, \norm{(G_{\mu}^{-1})^\prime}_\infty, \norm{(G_{\mu}^{-1})^{\prime \prime}}_{\infty} < \infty.
\end{flalign*}
Moreover, we will often use $t_i - t_{i-1} \leq \frac{2}{n}$ in the following. Let $L_{\widetilde{\mu}} > 0$ be the Lipschitz-constant of $\widetilde{\mu}$ and $L_{\widetilde{\sigma}}> 0$ be the Lipschitz-constant of $\widetilde{\sigma}$. We start wih the estimation of $\EE|A_{z, \tilde{z}}|^2$. It holds
\begin{flalign*}
\EE|A_{z, \tilde{z}}|^2 &\leq \frac{2}{n} \cdot \norm{(G_{\mu}^{-1})^\prime}_\infty^2 \EE\int_0^{t_i - t_{i-1}} |\widetilde{\mu}(\widetilde{Y}^{G_{\mu}(\tilde{z})}_s) - \widetilde{\mu}(\widetilde{Y}_s^{G_{\mu}(\tilde{z})} - G_{\mu}(\tilde{z}) + G_{\mu}(z))|^2 \: ds
\\& \leq \frac{4}{n^2} \cdot \norm{(G_{\mu}^{-1})^\prime}_\infty^2 \cdot (L_{\widetilde{\mu}})^2 | G_{\mu}(z) - G_{\mu}(\tilde{z})|^2.
\end{flalign*}
In a similar fashion, we derive for $\EE|B_{z, \tilde{z}}|^2$ 
\begin{flalign*}
\EE|B_{z, \tilde{z}}|^2 & \leq  \frac{1}{2n} \norm{(G_{\mu}^{-1})^{\prime \prime}}_\infty^2\cdot   \EE\int_0^{t_i - t_{i-1}} |\widetilde{\sigma}(\widetilde{Y}^{G_{\mu}(\tilde{z})}_s) + \widetilde{\sigma}(\widetilde{U}^{z,\tilde{z}}_s)|^2 \cdot |\widetilde{\sigma}(\widetilde{Y}^{G_{\mu}(\tilde{z})}_s) - \widetilde{\sigma}(\widetilde{U}^{z,\tilde{z}}_s)|^2 \: ds
\\& \leq \frac{4}{n^2} \cdot  \norm{(G_{\mu}^{-1})^{\prime \prime}}_\infty^2 \cdot \norm{\widetilde{\sigma}}_{\infty}^2 \cdot (L_{\widetilde{\sigma}})^2 |G(z) - G(\tilde{z})|^2.
\end{flalign*}
Also the expression $\EE|C_{z, \tilde{z}}|^2$ can be estimated in a similar manner after an application of the It\^{o}-isometry by
\begin{flalign*}
\EE |C_{z,\tilde{z}}|^2 &\leq \norm{(G_{\mu}^{-1})^\prime}_\infty^2 \cdot \EE \int_0^{t_i - t_{i-1}} |\widetilde{\sigma}(\widetilde{Y}^{G_{\mu}(\tilde{z})}_s) - \widetilde{\sigma}(\widetilde{Y}_s^{G_{\mu}(\tilde{z})} - G_{\mu}(\tilde{z}) + G_{\mu}(z))|^2 \: ds 
\\& \leq \frac{2}{n} \cdot \norm{(G_{\mu}^{-1})^\prime}_\infty^2 \cdot (L_{\widetilde{\sigma}})^2 |G_{\mu}(z) - G_{\mu}(\tilde{z})|^2.
\end{flalign*}
\end{proof}

We continue with the derivation of a lower bound for \eqref{eqn:lowerBound:increments:statement}. Since we will therefore apply Lemma \ref{lemma:boundProbability:SolutionIsInInterval} to estimate the probability that $X_{t_{i-1}}$ stays in some small interval, we will bound expression \eqref{eqn:lowerBound:increments:statement} in the following lemma only for such $i \in \{1, \dots, n\}$ which satisfy $t_i > 1/2$.

\begin{lemma}\label{lemma:lowerBound:problemReductionToDistanceOfTildeProcesses}
Let $\mu \colon \mathbb{R} \rightarrow \mathbb{R}$ be a function satisfying $(\mu 1)$ and $(\mu 2)$. Let $x_0 \in \mathbb{R}$ and $X, \widetilde{X}\colon [0,1] \times \Omega \rightarrow \mathbb{R}$ be strong solutions of the SDE \eqref{eqn:basicSDE} with initial value $x_0$ and driving Brownian motion $W$ and $\widetilde{W}$, respectively. Then there exist constants $c,C_1, C_2 > 0$ such that for all $i \in \{1, \dots, n\}$ with $t_i > \frac{1}{2}$ it holds
\begin{flalign*}
\EE &|\int_{t_{i-1}}^{t_i} \mu(X_s) - \mu(G_{\mu}^{-1}(G_{\mu}(\widetilde{X}_s) - G_{\mu}(\widetilde{X}_{t_{i-1}}) + G_{\mu}(X_{t_{i-1}}))) \: ds|^2 
\\&\geq c \EE |\int_{t_{i-1}}^{t_i} \mu(X_s) - \mu(\widetilde{X}_{s}) \: ds|^2 - \frac{C_1}{n} \EE |G_{\mu}(X_{t_{i-1}}) - G_{\mu}(\widetilde{X}_{t_{i-1}})|^2 - \frac{C_2}{n^{5 \slash 2 + 1 \slash 16}}.
\end{flalign*}
\end{lemma}

\begin{proof}
Let $i \in \{1, \dots, n\}$ with $t_i > \frac{1}{2}$. First of all, we use $\frac{1}{2}(a+b)^2 - b^2 \leq a^2$ for $a,b \in \mathbb{R}$ to obtain
\begin{flalign*}
\EE &|\int_{t_{i-1}}^{t_i} \mu(X_s) - \mu(G_{\mu}^{-1}(G_{\mu}(\widetilde{X}_s) - G_{\mu}(\widetilde{X}_{t_{i-1}}) + G_{\mu}(X_{t_{i-1}}))) \: ds|^2 
\\&\geq \frac{1}{2}\EE |\int_{t_{i-1}}^{t_i} \mu(X_s) - \mu(\widetilde{X}_{s}) \: ds|^2
\\& \quad  - \EE |\int_{t_{i-1}}^{t_i} \mu(\widetilde{X}_{s}) - \mu(G_{\mu}^{-1}(G_{\mu}(\widetilde{X}_s) - G_{\mu}(\widetilde{X}_{t_{i-1}}) + G_{\mu}(X_{t_{i-1}}))) \: ds|^2
\end{flalign*}
Hence, it is sufficient to show the existence of constants $C_1, C_2 > 0$ with
\begin{flalign*}
\EE &|\int_{t_{i-1}}^{t_i} \mu(\widetilde{X}_{s}) - \mu(G_{\mu}^{-1}(G_{\mu}(\widetilde{X}_s) - G_{\mu}(\widetilde{X}_{t_{i-1}}) + G_{\mu}(X_{t_{i-1}}))) \: ds|^2 
\\&\leq \frac{C_1}{n} \EE |G_{\mu}(X_{t_{i-1}}) - G_{\mu}(\widetilde{X}_{t_{i-1}})|^2 + \frac{C_2}{n^{5 \slash 2 + 1 \slash 16}}.
\end{flalign*}
Therefore, we note that according to Lemma \ref{lemma:functionalsWelldefinedMuOfLinearGrowth} we can fix a Lipschitz-continuous function $\mu_{cont} \colon \mathbb{R} \rightarrow \mathbb{R}$ and real numbers $\alpha_1, \dots, \alpha_k, \gamma_1, \dots, \gamma_k \in \mathbb{R}$ such that
\begin{flalign*}
\mu = \mu_{cont} + \sum_{i=1}^k \alpha_i 1_{[\xi_i, \infty)} + \sum_{i=1}^k \gamma_i 1_{\{\xi_i\}}.
\end{flalign*}
In the following we will derive for $\xi \in \mathbb{R}$ constants $C_1^{cont}, C_2^{\xi} > 0$ such that 
\begin{flalign}
\begin{aligned}
\EE &|\int_{t_{i-1}}^{t_i} \mu_{cont}(\widetilde{X}_{s}) - \mu_{cont}(G_{\mu}^{-1}(G_{\mu}(\widetilde{X}_s) - G_{\mu}(\widetilde{X}_{t_{i-1}}) + G_{\mu}(X_{t_{i-1}}))) \: ds|^2
\\& \leq \frac{C_1^{cont}}{n^2} \EE |G_{\mu}(X_{t_{i-1}}) - G_{\mu}(\widetilde{X}_{t_{i-1}})|^2
\end{aligned} \label{eqn:upperBound:increments:continuous}
\end{flalign} 
and
\begin{flalign}
\EE &|\int_{t_{i-1}}^{t_i} 1_{[\xi, \infty)}(\widetilde{X}_{s}) - 1_{[\xi, \infty)}(G_{\mu}^{-1}(G_{\mu}(\widetilde{X}_s) - G_{\mu}(\widetilde{X}_{t_{i-1}}) + G_{\mu}(X_{t_{i-1}}))) \: ds|^2 \leq \frac{C_2^{\xi}}{n^{5 \slash 2 + 1 \slash 16}} \label{eqn:upperBound:increments:discontinuous}
\end{flalign}
as well as 
\begin{flalign}
\EE &|\int_{t_{i-1}}^{t_i} 1_{\{\xi\}}(\widetilde{X}_{s}) - 1_{\{\xi\}}(G_{\mu}^{-1}(G_{\mu}(\widetilde{X}_s) - G_{\mu}(\widetilde{X}_{t_{i-1}}) + G_{\mu}(X_{t_{i-1}}))) \: ds|^2 = 0. \label{eqn:upperBound:increments:discontinuous:jumpPoints}
\end{flalign}
Having done so, the claim follows immediately.

We now turn our attention to the derivation of \eqref{eqn:upperBound:increments:continuous}. Let $L_{\mu_{cont}} > 0$ be the Lipschitz-constant of $\mu_{cont}$ and let $L_{G_{\mu}^{-1}} > 0$ be the Lipschitz-constant of $G_{\mu}^{-1}$ which exists due to Lemma \ref{lemma:basicPropertiesOfG}. Then it follows from the H\"older-inequality and in consideration of $t_i - t_{i-1} \leq \frac{2}{n}$
\begin{flalign*}
\EE &|\int_{t_{i-1}}^{t_i} \mu_{cont}(\widetilde{X}_{s}) - \mu_{cont}(G_{\mu}^{-1}(G_{\mu}(\widetilde{X}_s) - G_{\mu}(\widetilde{X}_{t_{i-1}}) + G_{\mu}(X_{t_{i-1}}))) \: ds|^2 
\\&\leq (t_i - t_{i-1}) \EE \int_{t_{i-1}}^{t_i} |\mu_{cont}(G_{\mu}^{-1}(G_{\mu}(\widetilde{X}_s))) - \mu_{cont}(G_{\mu}^{-1}(G_{\mu}(\widetilde{X}_s) - G_{\mu}(\widetilde{X}_{t_{i-1}}) + G_{\mu}(X_{t_{i-1}}))) |^2 \: ds 
\\&\leq \frac{4}{n^2} \cdot (L_{\mu_{cont}})^2 \cdot (L_{G_{\mu}^{-1}})^2 \cdot \EE |G_{\mu}(X_{t_{i-1}}) - G_{\mu}(\widetilde{X}_{t_{i-1}})|^2.
\end{flalign*}
Hence, we have found the desired upper bound in \eqref{eqn:upperBound:increments:continuous} and we will now derive the existence of $C_2^{\xi} > 0$ from \eqref{eqn:upperBound:increments:discontinuous} with $\xi \in \mathbb{R}$. Let $\xi \in \mathbb{R}$. We proceed similar to the proof of Lemma 14 in~\cite{MGY21}. Since $G_{\mu}$ is strictly monotonically increasing due to Lemma \ref{lemma:basicPropertiesOfG}, it follows with the H\"older-inequality that
\begin{flalign*}
\EE &|\int_{t_{i-1}}^{t_i} 1_{[\xi, \infty)}(\widetilde{X}_{s}) - 1_{[\xi, \infty)}(G_{\mu}^{-1}(G_{\mu}(\widetilde{X}_s) - G_{\mu}(\widetilde{X}_{t_{i-1}}) + G_{\mu}(X_{t_{i-1}}))) \: ds|^2
\\&\leq (t_i - t_{i-1}) \EE \int_{t_{i-1}}^{t_i} |1_{[G_{\mu}(\xi), \infty)}(G_{\mu}(\widetilde{X}_{s})) - 1_{[G_{\mu}(\xi), \infty)}(G_{\mu}(\widetilde{X}_s) - G_{\mu}(\widetilde{X}_{t_{i-1}}) + G_{\mu}(X_{t_{i-1}}))|^2 \: ds
\end{flalign*}
Since for all $x,z \in \mathbb{R}$ there holds the estimation
\begin{flalign*}
|1_{[G_{\mu}(\xi), \infty)}(x) - 1_{[G_{\mu}(\xi), \infty)}(x+z)| &= 1_{\{x < G_{\mu}(\xi) \leq x + z\}} + 1_{\{x + z < G_{\mu}(\xi) \leq x\}} 
\\&= 1_{\{0 < G_{\mu}(\xi) - x \leq z\}} + 1_{\{z < G_{\mu}(\xi) - x \leq 0\}} \leq 1_{\{|G_{\mu}(\xi) - x| \leq |z|\}}, 
\end{flalign*}
it follows that
\begin{flalign}
\begin{aligned}
\EE &|\int_{t_{i-1}}^{t_i} 1_{[\xi, \infty)}(\widetilde{X}_{s}) - 1_{[\xi, \infty)}(G_{\mu}^{-1}(G_{\mu}(\widetilde{X}_s) - G_{\mu}(\widetilde{X}_{t_{i-1}}) + G_{\mu}(X_{t_{i-1}}))) \: ds|^2 
\\&\leq (t_i - t_{i-1}) \int_{t_{i-1}}^{t_i} \mathbb{P}(|G_{\mu}(\widetilde{X}_{s}) - G_{\mu}(\xi)| \leq | G_{\mu}(X_{t_{i-1}}) - G_{\mu}(\widetilde{X}_{t_{i-1}}) |) \: ds. 
\end{aligned}\label{eqn:upperBound:increments:discontinuous:indicator}
\end{flalign}
Fix $\varepsilon \in (0, \infty)$. Then we have  for $s \in [t_{i-1}, t_i]$
\begin{flalign}
\begin{aligned}
\mathbb{P}&(|G_{\mu}(\widetilde{X}_{s}) - G_{\mu}(\xi)| \leq | G_{\mu}(X_{t_{i-1}}) - G_{\mu}(\widetilde{X}_{t_{i-1}}) |) 
\\&\leq \mathbb{P}(|G_{\mu}(\widetilde{X}_{s}) - G_{\mu}(\xi)| \leq n^{-1/2 - \varepsilon}) + \mathbb{P}(n^{-1/2 - \varepsilon} \leq | G_{\mu}(X_{t_{i-1}}) - G_{\mu}(\widetilde{X}_{t_{i-1}}) |). \
\end{aligned}\label{eqn:upperBound:increments:discontinuous:indicator:separateProb}
\end{flalign}
Further, we will investigate the last-mentioned probabilities in more detail and we will derive constants $D_1^{\xi}, D_2^{\xi} > 0$ with
\begin{flalign}
\mathbb{P}(|G_{\mu}(\widetilde{X}_{s}) - G_{\mu}(\xi)| \leq n^{-1/2 - \varepsilon}) \leq \frac{D_1^{\xi}}{n^{1/2 + \varepsilon}}, \quad s \in [t_{i-1}, t_i], \label{eqn:upperBound:increments:discontinuous:indicator:prob1}
\end{flalign}
as well as
\begin{flalign}
\mathbb{P}(n^{-1/2 - \varepsilon} \leq | G_{\mu}(X_{t_{i-1}}) - G_{\mu}(\widetilde{X}_{t_{i-1}}) |) \leq \frac{D_2^{\xi}}{n^{3/4 - 3\varepsilon}} \label{eqn:upperBound:increments:discontinuous:indicator:prob2}.
\end{flalign}
Plugging \eqref{eqn:upperBound:increments:discontinuous:indicator:prob1} and \eqref{eqn:upperBound:increments:discontinuous:indicator:prob2} into \eqref{eqn:upperBound:increments:discontinuous:indicator:separateProb} together with  \eqref{eqn:upperBound:increments:discontinuous:indicator}, $t_i - t_{i-1} \leq \frac{2}{n}$ and the choice $\varepsilon = 1/16$ provides the desired estimation in \eqref{eqn:upperBound:increments:discontinuous}. Let us start now with the derivation of \eqref{eqn:upperBound:increments:discontinuous:indicator:prob1}. Since $G_{\mu}^{-1}$ is strictly monotonically increasing due to Lemma \ref{lemma:basicPropertiesOfG}, it holds for $s \in [t_{i-1}, t_i]$
\begin{flalign*}
\mathbb{P}(|G_{\mu}(\widetilde{X}_{s}) - G_{\mu}(\xi)| \leq n^{-1/2 - \varepsilon})  &= \mathbb{P}(-n^{-1/2 - \varepsilon} \leq G_{\mu}(\widetilde{X}_{s}) - G_{\mu}(\xi) \leq n^{-1/2 - \varepsilon})
\\& = \mathbb{P}(G_{\mu}^{-1}(-n^{-1/2 - \varepsilon} + G_{\mu}(\xi)) \leq \widetilde{X}_{s} \leq G_{\mu}^{-1}(n^{-1/2 - \varepsilon} + G_{\mu}(\xi)) ).
\end{flalign*}
Because of $t_i > 1/2$, Lemma \ref{lemma:functionalsWelldefinedMuOfLinearGrowth} and Lemma \ref{lemma:boundProbability:SolutionIsInInterval} show the existence of a constant $d_1^{\xi} > 0$ which is independent of $n$ and which satisfies
\begin{flalign*}
\mathbb{P}&(G_{\mu}^{-1}(-n^{-1/2 - \varepsilon} + G_{\mu}(\xi)) \leq \widetilde{X}_{s} \leq G_{\mu}^{-1}(n^{-1/2 - \varepsilon} + G_{\mu}(\xi)) ) 
\\&\leq d_1^{\xi} (G_{\mu}^{-1}(n^{-1/2 - \varepsilon} + G_{\mu}(\xi))  - G_{\mu}^{-1}(-n^{-1/2 - \varepsilon} + G_{\mu}(\xi)) ).
\end{flalign*}
Since $G_{\mu}^{-1}$ is Lipschitz-continuous due to Lemma \ref{lemma:basicPropertiesOfG}, we obtain the existence of the constant $D_1^{\xi}$ in \eqref{eqn:upperBound:increments:discontinuous:indicator:prob1}. Next, we derive the constant $D_2^{\xi}$ in \eqref{eqn:upperBound:increments:discontinuous:indicator:prob2}. Therefore, let $L_{G_{\mu}} > 0$ be the Lipschitz-constant of $G_{\mu}$ which exists due to Lemma \ref{lemma:basicPropertiesOfG}. Then we obtain with an application of the Markov-inequality
\begin{flalign*}
\mathbb{P}(n^{-1/2 - \varepsilon} &\leq | G_{\mu}(X_{t_{i-1}}) - G_{\mu}(\widetilde{X}_{t_{i-1}}) |) 
\\&\leq \frac{1}{n^{-3/2 - 3\varepsilon}} \EE |G_{\mu}(X_{t_{i-1}}) - G_{\mu}(\widetilde{X}_{t_{i-1}}) |^3 \leq (L_{G_{\mu}})^3 n^{3/2 + 3\varepsilon} \cdot \EE |X_{t_{i-1}} - \widetilde{X}_{t_{i-1}} |^3.
\end{flalign*}
The existence of $D_2^{\xi}$ immediately follows with Lemma \ref{lemma:boundLPDistanceOfSolutionsAtInterpolation} since $(\mu 1)$ and ($\mu 2$) hold.

Now let us show \eqref{eqn:upperBound:increments:discontinuous:jumpPoints} for some $\xi \in \mathbb{R}$. It holds due to the bijectivity of $G_\mu$
\begin{flalign*}
\EE &|\int_{t_{i-1}}^{t_i} 1_{\{\xi\}}(\widetilde{X}_{s}) - 1_{\{\xi\}}(G_{\mu}^{-1}(G_{\mu}(\widetilde{X}_s) - G_{\mu}(\widetilde{X}_{t_{i-1}}) + G_{\mu}(X_{t_{i-1}}))) \: ds|^2
\\& \leq 2 \int_{t_{i-1}}^{t_i}  \mathbb{P}(\widetilde{X}_s = \xi) + \mathbb{P}(G_{\mu}(\widetilde{X}_s) = G_\mu(\xi) + G_{\mu}(\widetilde{X}_{t_{i-1}}) - G_{\mu}(X_{t_{i-1}})) \: ds.
\end{flalign*}
An application of Lemma \ref{lemma:functionalsWelldefinedMuOfLinearGrowth} and Lemma \ref{lemma:boundProbability:SolutionIsInInterval} shows that $\mathbb{P}(\widetilde{X}_s = \xi)  = 0$ for all $s \in (0,1]$. Hence, it suffices to prove that it also holds $\mathbb{P}(G_{\mu}(\widetilde{X}_s) = G_\mu(\xi) + G_{\mu}(\widetilde{X}_{t_{i-1}}) - G_{\mu}(X_{t_{i-1}})) = 0$ for all $s \in (t_{i-1}, t_i]$. Let $s \in (t_{i-1}, t_i]$ and let $F$ as well as $\widetilde{V}$ be given as in Lemma \ref{lemma:representationOfXWithF}. Then we have
\begin{flalign*}
\mathbb{P}&(G_{\mu}(\widetilde{X}_s) = G_\mu(\xi) + G_{\mu}(\widetilde{X}_{t_{i-1}}) - G_{\mu}(X_{t_{i-1}}))
\\ &= \mathbb{P}(G_{\mu}(F(\widetilde{X}_{t_{i-1}}, \widetilde{V})(s - t_{i-1})) = G_\mu(\xi) + G_{\mu}(\widetilde{X}_{t_{i-1}}) - G_{\mu}(X_{t_{i-1}})).
\end{flalign*}
Since $(X_{t_{i-1}}, \widetilde{X}_{t_{i-1}})$ and $\widetilde{V}$ are independent it thus suffices to show that it holds for $\mathbb{P}^{(X_{t_{i-1}}, \widetilde{X}_{t_{i-1}})}$-almost all $(z, \tilde{z}) \in \mathbb{R}^2$
\begin{flalign}
\mathbb{P}(G_{\mu}(F(\tilde{z}, \widetilde{V})(s - t_{i-1})) = G_\mu(\xi) + G_{\mu}(\tilde{z}) - G_{\mu}(z)) = 0. \label{eqn:upperBound:increments:discontinuous:jumpPoints:conditoned}
\end{flalign}
By Lemma \ref{lemma:existenceOfFunctionF} the process $F(\tilde{z}, \widetilde{V})$ is a strong solution of the SDE \eqref{eqn:basicSDE} with initial value $\tilde{z}$ and driving Brownian motion $\widetilde{V}$. Since $\mu$ satisfies the linear growth property according to \mbox{Lemma \ref{lemma:functionalsWelldefinedMuOfLinearGrowth},} the random variable $F(\tilde{z}, \widetilde{V})(s - t_{i-1})$ thus has a Lebesgue-density according to Theorem 2.1 in~\cite{FP10}. Hence, we obtain \eqref{eqn:upperBound:increments:discontinuous:jumpPoints:conditoned} which concludes the proof.
\end{proof}

\begin{lemma}\label{lemma:lowerBound:dropRecursionIntegralOfSDE}
Let $\mu \colon \mathbb{R} \rightarrow \mathbb{R}$ be a function satisfying $(\mu 1)$ and $(\mu 2)$. Let $x_0 \in \mathbb{R}$ and $X,\widetilde{X} \colon [0,1] \times \Omega \rightarrow \mathbb{R}$ be strong solutions of the SDE \eqref{eqn:basicSDE} on the time interval $[0,1]$ with initial value $x_0$ and driving Brownian motion $W$ and $\widetilde{W}$, respectively. Then there exists a constant $c > 0$ such that for all $i \in \{1, \dots, n\}$ with $t_{i} > \frac{1}{2}$ it holds
\begin{flalign*}
\EE &|\int_{t_{i-1}}^{t_i} (\mu(X_s) - \mu(\widetilde{X}_s)) \: ds|^2
\\& \geq \frac{1}{16} \EE | \int_{t_{i-1}}^{t_i} (\mu(X_{t_{i-1}} + W_s - W_{t_{i-1}}) - \mu(X_{t_{i-1}} + \widetilde{W}_s - \widetilde{W}_{t_{i-1}})) \: ds |^2 - \frac{c}{n^{5 \slash 2 + 1 \slash 16}}.
\end{flalign*}
\end{lemma}

\begin{proof}
We will use some arguments of the proof of Lemma 14 in~\cite{MGY21} and similar to this proof we set for $i \in \{1, \dots, n\}$
\begin{flalign*}
A_i &= \int_{t_{i-1}}^{t_i} (\mu(X_{t_{i-1}} + W_t - W_{t_{i-1}}) - \mu(X_{t_{i-1}} + \widetilde{W}_t - \widetilde{W}_{t_{i-1}})) \: dt,
\\
B_i &= \int_{t_{i-1}}^{t_i} (\mu(X_{t_{i-1}} + W_t - W_{t_{i-1}}) - \mu(X_t)) \: dt,
\\
C_i &= \int_{t_{i-1}}^{t_i} (\mu(X_t) - \mu(\widetilde{X}_t)) \: dt,
\\
D_i &= \int_{t_{i-1}}^{t_i} (\mu(\widetilde{X}_t) - \mu(\widetilde{X}_{t_{i-1}} + \widetilde{W}_t - \widetilde{W}_{t_{i-1}})) \: dt,
\\
E_i &= \int_{t_{i-1}}^{t_i} (\mu(\widetilde{X}_{t_{i-1}} + \widetilde{W}_t - \widetilde{W}_{t_{i-1}}) - \mu(X_{t_{i-1}} + \widetilde{W}_t - \widetilde{W}_{t_{i-1}})) \: dt.
\end{flalign*}
Because of $A_i = B_i + C_i + D_i + E_i$, we immediately obtain
\begin{flalign*}
\EE |A_i|^2 \leq 16(\EE|B_i|^2 + \EE|C_i|^2 + \EE|D_i|^2 + \EE|E_i|^2) = 16(2\EE|B_i|^2 + \EE|C_i|^2 + \EE|E_i|^2).
\end{flalign*}
So we see that it suffices to show the existence of a constant $C > 0$ such that for all $i \in \{1, \dots, n\}$ with $t_i > \frac{1}{2}$ it holds
\begin{flalign*}
\max\{\EE|B_i|^2, \EE|E_i|^2\} \leq \frac{C}{n^{5 \slash 2 + 1 \slash 16}}.
\end{flalign*}
Therefore, let $i \in \{1, \dots, n\}$ with $t_i > \frac{1}{2}$. We note that all constants which are derived in the following are independent of $i$. According to Lemma \ref{lemma:functionalsWelldefinedMuOfLinearGrowth} there exists a constant $c_1 > 0$ such that for all $t \in [t_{i-1}, t_i]$ it holds
\begin{flalign*}
|\mu&(X_{t_{i-1}} + W_t - W_{t_{i-1}}) - \mu(X_t)| 
\\&\leq c_1(|X_{t_{i-1}} + W_t - W_{t_{i-1}} - X_t| + \sum_{j=1}^k 1_{\{(X_{t_{i-1}} + W_t - W_{t_{i-1}}- \xi_j)(X_t - \xi_j) \leq 0\}})
\\& = c_1(|\int_{t_{i-1}}^t \mu(X_s) \: ds| + \sum_{j=1}^k 1_{\{(X_{t_{i-1}} + W_t - W_{t_{i-1}}- \xi_j)(X_t - \xi_j) \leq 0\}}).
\end{flalign*}
This yields the existence of a constant $c_2 > 0$ with
\begin{flalign*}
\EE|B_i|^2 &\leq c_2 \big( \EE\big[|\int_{t_{i-1}}^{t_i} |\int_{t_{i-1}}^t \mu(X_s) \: ds| \: dt|^2\big] + \sum_{j=1}^k \EE | \int_{t_{i-1}}^{t_i}1_{\{(X_{t_{i-1}} + W_t - W_{t_{i-1}}- \xi_j)(X_t - \xi_j) \leq 0\}} \: dt |^2 \big)
\\& \leq c_2 \big( \EE\big[(t_i - t_{i-1})^4 \cdot \sup_{t\in [0,1]} |\mu(X_t)|^2 \big] 
+ \sum_{j=1}^k \EE | \int_{t_{i-1}}^{t_i}1_{\{(X_{t_{i-1}} + W_t - W_{t_{i-1}}- \xi_j)(X_t - \xi_j) \leq 0\}}\: dt |^2 \big).
\end{flalign*}
Since $\mu$ satisfies the linear growth property according to Lemma \ref{lemma:functionalsWelldefinedMuOfLinearGrowth}, we obtain a constant $c_3 > 0$ with $\EE [\sup_{t\in [0,1]} |\mu(X_t)|^2] \leq c_3$. Due to $t_i - t_{i-1}  \leq \frac{2}{n}$ and Lemma \ref{lemma:lowerBound:dropRecursionIntegralOfSDE:boundJumpPositios}, there thus exists a constant $c_4 > 0$ with 
\begin{flalign*}
\EE|B_i|^2 \leq \frac{c_4}{n^{5 \slash 2+ 1 \slash  6}}.
\end{flalign*}
So, we have found the desired bound for $\EE|B_i|^2$ and we will subsequently derive the corresponding bound for $\EE|E_i|^2$. It holds
\begin{flalign*}
\EE|E_i|^2 \leq (t_i - t_{i-1}) \int_{t_{i-1}}^{t_i} \EE |\mu(\widetilde{X}_{t_{i-1}} + \widetilde{W}_t - \widetilde{W}_{t_{i-1}}) - \mu(X_{t_{i-1}} + \widetilde{W}_t - \widetilde{W}_{t_{i-1}})|^2 \: dt.
\end{flalign*}
Since the random variables $(X_{t_{i-1}}, \widetilde{X}_{t_{i-1}})$ and $(\widetilde{W}_t - \widetilde{W}_{t_{i-1}})_{t \in [t_{i-1}, t_i]}$ are independent according to Lemma \ref{lemma:representationOfXWithF}, we obtain
\begin{flalign*}
\EE|E_i|^2 \leq (t_i - t_{i-1}) \int_{t_{i-1}}^{t_i} \int_{\mathbb{R}} \EE |\mu(\widetilde{X}_{t_{i-1}} +u) - \mu(X_{t_{i-1}} + u)|^2 \: \mathbb{P}^{\widetilde{W}_t - \widetilde{W}_{t_{i-1}}} (du) \: dt.
\end{flalign*}
Applying Lemma \ref{lemma:functionalsWelldefinedMuOfLinearGrowth} yields the existence of a constant $c_5 > 0$ such that for all $u \in \mathbb{R}$ it holds
\begin{flalign*}
|\mu(X_{t_{i-1}} + u) - \mu(\widetilde{X}_{t_{i-1}} +u)| \leq c_5(|X_{t_{i-1}} - \widetilde{X}_{t_{i-1}}| + \sum_{j=1}^k 1_{\{(X_{t_{i-1}} + u - \xi_j)(\widetilde{X}_{t_{i-1}} + u - \xi_j) \leq 0\}}).
\end{flalign*}
Hence, there exists a constant $c_6 > 0$ such that
\begin{flalign}
\begin{aligned}
\EE|E_i|^2 &\leq c_6((t_i - t_{i-1})^2\EE|X_{t_{i-1}} - \widetilde{X}_{t_{i-1}}|^2 
\\& \quad+ (t_i - t_{i-1}) \sum_{j=1}^k \int_{t_{i-1}}^{t_i}\int_{\mathbb{R}}\EE|1_{\{(X_{t_{i-1}} + u - \xi_j)(\widetilde{X}_{t_{i-1}} + u - \xi_j) \leq 0\}}|\: \mathbb{P}^{\widetilde{W}_t - \widetilde{W}_{t_{i-1}}} (du)\: dt).
\end{aligned} \label{eqn:lowerBound:dropRecursionIntegralOfSDE:boundEi}
\end{flalign}
We will later show the existence of a constant $c_7 > 0$ such that
\begin{flalign}
(t_i - t_{i-1}) \sum_{j=1}^k \int_{t_{i-1}}^{t_i}\int_{\mathbb{R}}\EE|1_{\{(X_{t_{i-1}} + u - \xi_j)(\widetilde{X}_{t_{i-1}} + u - \xi_j) \leq 0\}}| \: \mathbb{P}^{\widetilde{W}_t - \widetilde{W}_{t_{i-1}}} (du) \: dt \leq \frac{c_7}{n^{5 \slash 2 + 1 \slash 16}}. \label{eqn:lowerBound:dropRecursionIntegralOfSDE:boundJumpParts}
\end{flalign}
Applying Lemma \ref{lemma:boundLPDistanceOfSolutionsAtInterpolation} yields the existence of a constant $c_8 > 0$ such that
\begin{flalign*}
\max_{i \in \{0, \dots, n\}} \EE |X_{t_{i-1}} - \widetilde{X}_{t_{i-1}}|^2 \leq \frac{c_8}{n^{3 \slash 2}}.
\end{flalign*}
Thus, we obtain in consideration of \eqref{eqn:lowerBound:dropRecursionIntegralOfSDE:boundEi} and \eqref{eqn:lowerBound:dropRecursionIntegralOfSDE:boundJumpParts} a constant $c_9 > 0$ such that it holds
\begin{flalign*}
\EE |E_i|^2 \leq \frac{c_9}{n^{5 \slash 2 + 1 \slash 16}}.
\end{flalign*}
Now the claim follows immediately.

It remains to prove \eqref{eqn:lowerBound:dropRecursionIntegralOfSDE:boundJumpParts}. Therefore, let $j \in \{1, \dots, k\}$ and $N \sim N(0,1)$. Then it holds
\begin{flalign*}
(t_i - t_{i-1})& \int_{t_{i-1}}^{t_i}\int_{\mathbb{R}}\EE|1_{\{(X_{t_{i-1}} + u - \xi_j)(\widetilde{X}_{t_{i-1}} + u - \xi_j) \leq 0\}}| \: \mathbb{P}^{\widetilde{W}_t - \widetilde{W}_{t_{i-1}}} (du) \: dt 
\\&\leq  (t_i - t_{i-1}) \int_{t_{i-1}}^{t_i}\int_{-1}^1 \EE|1_{\{(X_{t_{i-1}} + u - \xi_j)(\widetilde{X}_{t_{i-1}} + u - \xi_j) \leq 0\}}| \: \mathbb{P}^{\widetilde{W}_t - \widetilde{W}_{t_{i-1}}} (du) \: dt
\\&  \quad +(t_i - t_{i-1})^2 \cdot \mathbb{P}(\sqrt{t_i - t_{i-1}}|N| \geq 1).
\end{flalign*} 
Due to $t_i - t_{i-1} \leq \frac{2}{n}$, we thus obtain
\begin{flalign*}
\mathbb{P}(\sqrt{t_i - t_{i-1}}|N| \geq 1) \leq \mathbb{P}(|N| \geq \frac{\sqrt{n}}{\sqrt{2}}) \leq 2 e^{- \frac{n}{4}}.
\end{flalign*}
The upper bound in \eqref{eqn:lowerBound:dropRecursionIntegralOfSDE:boundJumpParts} now follows analogously to inequality (88) in the proof of Lemma 14 in~\cite{MGY21} where one has to observe that the set $\cup_{j=1}^k [\xi_j - 2, \xi_j + 2]$ is bounded and that one can thus use Lemma \ref{lemma:functionalsWelldefinedMuOfLinearGrowth} together with Lemma \ref{lemma:boundProbability:SolutionIsInInterval}.
\end{proof}

We now show a result which is originally proven in the Master's thesis~\cite{Ellinger2022}. The lemma is essential to drop the monotonicity condition on $\mu$. 

\begin{lemma}\label{lemma:lowerBoundForSquaredIntegral}
Let $(\Omega, \mathcal{F}, \mathbb{P})$ be a probability space, $t \in (0,1]$, $B,B^\prime \colon [0,t] \times \Omega \rightarrow \mathbb{R}$ be Brownian bridges on $[0,t]$ and let $U,V \colon \Omega \rightarrow \mathbb{R}$ be square-integrable random variables such that $B,B',U, V$ are independent and such that it holds $V \sim N(0, \frac{1}{t})$. Let $\mu_{disc} \colon \mathbb{R} \rightarrow \mathbb{R}$ be a step function satisfying $(\mu 1)$ and $(\mu 3)$ and let $i*\in \{1,\dots, k\}$ with $\mu_{disc}(\xi_{i*} +) \neq \mu_{disc}(\xi_{i*} -)$ as in $(\mu 3)$. Then there exist constants $c,d,D > 0$ which are independent of $U,V,B,B^\prime$ and $t$ such that
\begin{flalign}
\begin{aligned}
\EE&|\int_0^t (\mu_{disc}(U+sV+B_s) - \mu_{disc}(U+sV+ B'_s)) \: ds|^2 
\\&\geq c t^2\mathbb{P}(U \in [\xi_{i*}, \xi_{i*} + \sqrt{t}])- D e^{- d \slash t}.
\end{aligned} \label{eqn:LowerBoundDiscPart}
\end{flalign}
\end{lemma}

\begin{proof}
Similar to the proof of Lemma 3 in~\cite{MGY21} it holds
\begin{flalign}
\begin{aligned}
\EE &|\int_0^t (\mu_{disc}(U+sV+B_s) - \mu_{disc}(U+sV+ B'_s)) \: ds|^2 
\\&= 2 \int_{\mathbb{R}} \int_\mathbb{R} \int_0^t \int_0^t \varphi_{\mu_{disc}}(r,s,u,v) \: ds \: dr \: \mathbb{P}^V(dv) \: \mathbb{P}^U(du)
\end{aligned} \label{eqn:RewritingExpectedValueWithVarphi}
\end{flalign}
where we set for $r,s \in [0,t]$ and $(u,v) \in \mathbb{R}^2$
\begin{flalign*}
\varphi_{\mu_{disc}}(r,s,u,v) := &\EE [\mu_{disc}(B_r + u + rv)\mu_{disc}(B_s + u + sv)] 
\\&  - \EE [\mu_{disc}(B_r + u + rv)]\EE[\mu_{disc}(B_s + u + sv)].
\end{flalign*}
Now let $u,v \in \mathbb{R}$, $r,s\in (0,t)$ with $r \neq s$ and let us define $f_{disc},g_{disc} \colon \mathbb{R} \rightarrow \mathbb{R}$ by
\begin{flalign*}
f_{disc}(x) = \mu_{disc}(\frac{\sqrt{r(t-r)}}{\sqrt{t}}x + u + rv), \: g_{disc}(x) = \mu_{disc}(\frac{\sqrt{s (t-s)}}{\sqrt{t}}x + u + sv), \quad x \in \mathbb{R}.
\end{flalign*}
We put
\\\centerline{$Z := \frac{\sqrt{t}}{\sqrt{r(t-r)}}B_r, \: Y := \frac{\sqrt{t}}{\sqrt{s(t-s)}}B_s$}
as well as
\\\centerline{$a_i := a_i(r,s,u,v) := (\xi_i - u - rv) \frac{\sqrt{t}}{\sqrt{r(t-r)}}, \: b_j := b_j(r,s,u,v) := (\xi_j - u - sv) \frac{\sqrt{t}}{\sqrt{s(t-s)}}$}
for $i,j \in \{1,\dots, k\}$ and obtain with the covariance function of the Brownian bridge
\begin{flalign*}
\rho := \rho(r,s)  &:= \EE [\frac{\sqrt{t}}{\sqrt{s(t-s)}}B_s\frac{\sqrt{t}}{\sqrt{r(t-r)}}B_r] =\frac{(t - \max(r,s))\min(r,s)}{\sqrt{s(t-s)}\sqrt{r(t-r)}} \in (0,1).
\end{flalign*}
Following the arguments of the proof of Lemma 17 in~\cite{MGY21} for step functions instead of monotone functions and observing that $f_{disc}(a_i +)= \mu_{disc}(\xi_i +)$, $f_{disc}(a_i -) =  \mu_{disc}(\xi_i -)$, $g_{disc}(b_j+) = \mu_{disc}(\xi_j+)$ and $g_{disc}(b_j -) = \mu_{disc}(\xi_j -)$ for $i,j \in \{1, \dots, k\}$, we obtain with $c_{i,j}^{disc} := (\mu_{disc}(\xi_i +) - \mu_{disc}(\xi_i -))(\mu_{disc}(\xi_j+)- \mu_{disc}(\xi_j -)) $ for $i,j \in \{1, \dots, k\}$
\begin{flalign*}
\varphi_{\mu_{disc}}(r,s,u,v) = \sum_{i=1}^k  \sum_{j = 1}^k c_{i,j}^{disc} \int_0^\rho \frac{1}{\sqrt{2\pi}}e^{- \frac{a_i^2}{2}} \cdot \frac{1}{\sqrt{2\pi(1-x^2)}} e^{- \frac{(b_j-a_ix)^2}{2(1-x^2)}}
 \: dx.
\end{flalign*} 
The non-negativity of the integrand thus yields
\begin{flalign}
\varphi_{\mu_{disc}}&(r,s,u,v) \notag
\\&\geq (\mu_{disc}(\xi_{i*} +) - \mu_{disc}(\xi_{i*} -))^2 \int_0^\rho \frac{1}{2\pi}e^{- \frac{a_{i*}^2}{2}} \cdot \frac{1}{\sqrt{(1-x^2)}} e^{- \frac{(b_{i*}-a_{i*}x)^2}{2(1-x^2)}} \: dx \label{eqn:LowerBoundDiscPoint}
\\&\quad 
+\sum_{i=1}^k  \sum_{\substack{j=1\\j \neq i}}^k c_{i,j}^{disc} \cdot \int_0^\rho \frac{1}{2\pi}e^{- \frac{a_i^2}{2}} \cdot \frac{1}{\sqrt{(1-x^2)}} e^{- \frac{(b_j-a_ix)^2}{2(1-x^2)}}\: dx.\label{eqn:LowerBoundUnequalPoints}
\end{flalign}
The remaining part of the proof will be divided into two steps. At first, we note that it can be shown in a similar fashion as in Lemma 3 of~\cite{MGY21} that there exists a constant $c_1 > 0$ such that
\begin{flalign}
\begin{aligned}
\int_{\mathbb{R}} \int_\mathbb{R}&\int_0^t \int_0^t \int_0^{\rho(r,s)} (\frac{1}{\sqrt{2\pi}}e^{- \frac{a_{i*}^2}{2}} \cdot \frac{1}{\sqrt{2\pi(1-x^2)}} e^{- \frac{(b_{i*}-a_{i*}x)^2}{2(1-x^2)}}) ((r,s,u,v))\: 
 \: dx \: ds \: dr\mathbb{P}^V(dv) \: \mathbb{P}^U(du) 
\\&\geq c_1 t^2\mathbb{P}(U \in [\xi_{i*}, \xi_{i*} + \sqrt{t}]) \mathbb{P}(V \in [0, 1 \slash \sqrt{t}]).\
\end{aligned}\label{eqn:lemmaBoundsForSquaredIntegral:LowerBound:disc}
\end{flalign}
Now we have $ (\mu_{disc}(\xi_{i*} +) - \mu_{disc}(\xi_{i*} -))^2 > 0$ and exploiting the fact that it holds $V \sim N(0, \frac{1}{t})$ shows for $N \sim N(0,1)$ that
\begin{flalign*}
\mathbb{P}(V \in [0, 1 \slash \sqrt{t}]) = \mathbb{P}(N \in [0,1]) > 0.
\end{flalign*} 
Combining these observations with \eqref{eqn:LowerBoundDiscPoint} yields the left expression in the lower bound \eqref{eqn:LowerBoundDiscPart}. Subsequently, we will derive the right expression in \eqref{eqn:LowerBoundDiscPart} by using \eqref{eqn:LowerBoundUnequalPoints}. For this, we will prove the existence of constants $D,d > 0$ such that  
\begin{flalign*}
\sum_{i=1}^k  \sum_{\substack{j=1\\j \neq i}}^k |c_{i,j}^{disc}| \cdot \int_{\mathbb{R}} \int_\mathbb{R} \int_0^t \int_0^t \int_0^{\rho(r,s)} \psi_{i,j}(r,s,u,v,x)\: dx\: ds \: dr \: \mathbb{P}^V(dv) \: \mathbb{P}^U(du) \leq  D e^{- d \slash t}
\end{flalign*}
where we set for the simplicity of notations for $i,j \in \{1, \dots, k \}$ and $x \in [0,1)$
\\\centerline{$\psi_{i,j}(r,s,u,v,x) := \frac{1}{2\pi}e^{- \frac{(a_i(r,s,u,v))^2}{2}} \cdot \frac{1}{\sqrt{(1-x^2)}} e^{- \frac{(b_j(r,s,u,v)-a_i(r,s,u,v)x)^2}{2(1-x^2)}}$.}
A combination of \eqref{eqn:RewritingExpectedValueWithVarphi} and \eqref{eqn:LowerBoundUnequalPoints} then shows the claim.
If it holds $k=1$, then the constants $D,d > 0$ trivially exist. So let us assume $k \geq 2$ and let $\sigma := \min \{|\xi_i - \xi_j| : i,j \in \{1, \dots, k\}, i \neq j\}$ be the minimal distance of two possible jump positions. Furthermore, let $i,j \in \{1, \dots, k\}$ with $i \neq j$. We split the above iterated integral into two parts and we will treat the new expressions then separately. It holds
\begin{flalign}
\int_{\mathbb{R}} &\int_\mathbb{R} \int_0^t \int_0^t \int_0^{\rho(r,s)} \psi_{i,j}(r,s,u,v,x)
 \: dx\: ds \: dr \: \mathbb{P}^V(dv) \: \mathbb{P}^U(du)\notag
\\ =& \int_{\mathbb{R}} \int_\mathcal{S} \int_0^t \int_0^t \int_0^{\rho(r,s)}  \psi_{i,j}(r,s,u,v,x)
 \: dx\: ds \: dr \: \mathbb{P}^V(dv) \: \mathbb{P}^U(du) \label{eqn:IteratedIntegralS}
\\&+\int_{\mathbb{R}} \int_{\mathbb{R}\backslash \mathcal{S}} \int_0^t \int_0^t \int_0^{\rho(r,s)}  \psi_{i,j}(r,s,u,v,x)
 \: dx\: ds \: dr \: \mathbb{P}^V(dv) \: \mathbb{P}^U(du),\label{eqn:IteratedIntegralSComplement}
\end{flalign}
where we set $\mathcal{S}:= \{ v \in \mathbb{R} : |v| \leq \frac{\sigma}{4t}\}$. As we will see, we can bound the integral for $v \in \mathcal{S}$ in a suitable way. For the case $v \notin \mathcal{S}$ we will exploit the assumption $V \sim N(0, \frac{1}{t})$ to bound the probability of the event $V \in \mathbb{R}\backslash \mathcal{S}$.
We now bound \eqref{eqn:IteratedIntegralSComplement}. Put $N := \sqrt{t} V$. Then it holds  $N \sim N(0, 1)$ and due to an elementary Gaussian tail bound we have 
\begin{flalign*}
\mathbb{P}(V \notin \mathcal{S}) = \mathbb{P}(|V| > \frac{\sigma}{4t}) = \mathbb{P}(|N| > \frac{\sigma}{4\sqrt{t}}) \leq 2e^{- \frac{\sigma^2}{32t}}.
\end{flalign*}
Therewith, we obtain 
\begin{flalign*}
&\int_{\mathbb{R}} \int_{\mathbb{R}\backslash \mathcal{S}} \int_0^t \int_0^t \int_0^{\rho(r,s)}  \psi_{i,j}(r,s,u,v,x)
 \: dx\: ds \: dr \: \mathbb{P}^V(dv) \: \mathbb{P}^U(du)
\\&\leq  \int_{\mathbb{R}} \int_{\mathbb{R}\backslash \mathcal{S}} \int_0^1 \int_0^1 \int_0^1 \frac{1}{2\pi} \cdot \frac{1}{\sqrt{(1-x^2)}}
 \: dx\: ds \: dr \: \mathbb{P}^V(dv) \: \mathbb{P}^U(du) = \frac{\arcsin(1)}{2\pi} \cdot \mathbb{P}(V \notin \mathcal{S})
 \leq e^{- \frac{\sigma^2}{32t}}.
\end{flalign*}
This upper bound for the expression in \eqref{eqn:IteratedIntegralSComplement} is already of the desired form. Finally, we show the corresponding upper bound for the expression in \eqref{eqn:IteratedIntegralS}. We recall that it holds $i \neq j$ and that we have defined for fixed $u\in \mathbb{R}, v \in \mathcal{S}$ and $r,s\in (0,t)$ with $r \neq s$ 
\\\centerline{$\rho = \rho(r,s), \:a_i = a_i(r,s,u,v) = (\xi_i - u - rv) \frac{\sqrt{t}}{\sqrt{r(t-r)}}, \: b_j = b_j(r,s,u,v) = (\xi_j - u - sv) \frac{\sqrt{t}}{\sqrt{s(t-s)}}$.}
Now, let us estimate the integrand in \eqref{eqn:IteratedIntegralS} and obtain for $x \in [0,\rho)$
\begin{flalign}
\begin{aligned}
&\frac{1}{2\pi}e^{- \frac{a_i^2}{2}} \cdot \frac{1}{\sqrt{(1-x^2)}} e^{- \frac{(b_j-a_ix)^2}{2(1-x^2)}} =  \frac{1}{2\pi} \cdot \frac{1}{\sqrt{(1-x^2)}} e^{ \frac{-b_j^2 + 2 b_j a_ix - a_i^2}{2(1-x^2)}}
\\& \leq \frac{1}{2\pi} \cdot \frac{1}{\sqrt{(1-x^2)}} e^{ \frac{-b_j^2 + (b_j^2 + a_i^2)x - a_i^2}{2(1-x^2)}}  
 = \frac{1}{2\pi} \cdot \frac{1}{\sqrt{(1-x^2)}} e^{ -\frac{(b_j^2 + a_i^2)}{2(1+x)}} \leq \frac{1}{2\pi} \cdot \frac{1}{\sqrt{(1-x^2)}} e^{ -\frac{b_j^2 + a_i^2}{4}}.
\end{aligned}\label{eqn:lemma:LowerBoundStepFct:intermediateEstimateIntegrand}
\end{flalign}
Next, we show the existence of a constant $d_1 > 0$ sucht that $b_j^2 + a_i^2 \geq \frac{d_1}{t}$ and such that $d_1$ only depends on $\sigma$. Optimizing yields immediately
\begin{flalign*}
\frac{t}{r(t-r)} \geq \frac{t}{\frac{t}{2}(t - \frac{t}{2})} = \frac{4}{t} \text{\quad and \quad} \frac{t}{s(t-s)} \geq \frac{4}{t}.
\end{flalign*}
Therefore, it holds 
\begin{flalign*}
b_j^2 + a_i^2 \geq ((\xi_i - u - rv)^2 + (\xi_j - u - sv)^2) \cdot \frac{4}{t}.
\end{flalign*}
We now show the validity of $(\xi_i - u - rv)^2 + (\xi_j - u - sv)^2 \geq \frac{\sigma^2}{16}$ by contradiction. For this, let us assume that it holds $|\xi_i - u - rv| \leq \frac{\sigma}{4}$ and $|\xi_j - u - sv| \leq \frac{\sigma}{4}$. Then it follows 
\begin{flalign*}
|(\xi_i - \xi_j) - (r-s)v| = |(\xi_i - u - rv) - (\xi_j - u -sv) | \leq |\xi_i - u - rv| + |\xi_j - u -sv| \leq \frac{\sigma}{2}.
\end{flalign*}
In consideration of the definition of $\sigma$ it thus holds
\begin{flalign*}
|r-s| |v| = |(r-s)v| \geq |\xi_i - \xi_j| - \frac{\sigma}{2} \geq \frac{\sigma}{2}
\end{flalign*}
and due to $|r-s| \leq t$ we obtain
\begin{flalign*}
t |v| \geq \frac{\sigma}{2} \text{\quad and \quad} |v| \geq \frac{\sigma}{2t}, \quad \text{respectively}.
\end{flalign*}
The last inequality implies $v \notin \mathcal{S}$.  Hence, we have shown that there exists a constant $d_1>0$ such that $b_j^2 + a_i^2 \geq \frac{d_1}{t}$. Combining this with \eqref{eqn:lemma:LowerBoundStepFct:intermediateEstimateIntegrand} yields
\begin{flalign*}
\frac{1}{2\pi}e^{- \frac{a_i^2}{2}} \cdot \frac{1}{\sqrt{(1-x^2)}} e^{- \frac{(b_j-a_ix)^2}{2(1-x^2)}}\leq \frac{1}{2\pi} \cdot \frac{1}{\sqrt{(1-x^2)}} e^{ -\frac{d_1}{4t}}.
\end{flalign*}
Plugging this into \eqref{eqn:IteratedIntegralS} finishes the proof.
\end{proof}

\begin{lemma}\label{lemma:lowerBound:squaredIntegral:incremetsSimplified}
Let $\mu \colon \mathbb{R} \rightarrow \mathbb{R}$ be a function satisfying $(\mu 1)$ and $(\mu 3)$. Let $x_0 \in \mathbb{R}$ and $X,\widetilde{X} \colon [0,1] \times \Omega \rightarrow \mathbb{R}$ be strong solutions of the SDE \eqref{eqn:basicSDE} on the time interval $[0,1]$ with initial value $x_0$ and driving Brownian motion $W$ and $\widetilde{W}$, respectively. Then there exist constants $c, C > 0$ such that for all $i \in \{1, \dots, n\}$ with $t_i > \frac{1}{2}$ it holds
\begin{flalign*}
\EE | \int_{t_{i-1}}^{t_i} (\mu(X_{t_{i-1}} + W_s - W_{t_{i-1}}) - \mu(X_{t_{i-1}} + \widetilde{W}_s - \widetilde{W}_{t_{i-1}})) \: ds |^2 \geq c(t_i - t_{i-1})^{5 \slash 2} - \frac{C}{n^3}.
\end{flalign*}
\end{lemma}

\begin{proof}
Let $i \in \{1, \dots, n\}$ with $t_i > \frac{1}{2}$. First of all, according to Lemma \ref{lemma:functionalsWelldefinedMuOfLinearGrowth} $\mu$ can be written as $\mu = \mu_{cont} + \mu_{disc}$ where $\mu_{cont}$ is a Lipschitz-continuous function and $\mu_{disc}$ is a step function satisfying $(\mu 1)$ and $(\mu 3)$. Using $(a+b)^2 \leq 2 (a^2 + b^2)$ for $a,b \in \mathbb{R}$ we thus obtain
\begin{flalign*}
\EE &|\int_{t_{i-1}}^{t_i} (\mu(X_{t_{i-1}} + W_s - W_{t_{i-1}}) - \mu(X_{t_{i-1}} + \widetilde{W}_s - \widetilde{W}_{t_{i-1}})) \: ds |^2
\\& \geq \frac{1}{2} \EE |\int_{t_{i-1}}^{t_i} (\mu_{disc}(X_{t_{i-1}} + W_s - W_{t_{i-1}}) - \mu_{disc}(X_{t_{i-1}} + \widetilde{W}_s - \widetilde{W}_{t_{i-1}})) \: ds |^2
\\& \quad - \EE |\int_{t_{i-1}}^{t_i} (\mu_{cont}(X_{t_{i-1}} + W_s - W_{t_{i-1}}) - \mu_{cont}(X_{t_{i-1}} + \widetilde{W}_s - \widetilde{W}_{t_{i-1}})) \: ds |^2.
\end{flalign*}
We will later show the existence of constants $c_1, c_2, C_1, C_2 > 0$ which are independent of $i$ and which satisfy for $V := (V_t = W_{t_{i-1} + t} - W_{t_{i-1}})_{t \in [0, t_i - t_{i-1}]}$ and $\widetilde{V} := (\widetilde{V}_t = \widetilde{W}_{t_{i-1} + t} - \widetilde{W}_{t_{i-1}})_{t \in [0, t_i - t_{i-1}]}$
\begin{flalign}
\EE &| \int_{0}^{t_i - t_{i-1}} (\mu_{cont}(X_{t_{i-1}} + V_s) - \mu_{cont}(X_{t_{i-1}} + \widetilde{V}_s)) \: ds |^2 \leq \frac{C_1}{n^3}, \label{eqn:lowerBound:incremets:contPrart}
\\\EE &| \int_{0}^{t_i - t_{i-1}} (\mu_{disc}(X_{t_{i-1}} + V_s) - \mu_{disc}(X_{t_{i-1}} + \widetilde{V}_s)) \: ds |^2 \geq c_1 (t_i - t_{i-1})^{5 \slash 2} - C_2 e^{-c_2n}. \label{eqn:lowerBound:incremets:discPrart}&
\end{flalign}
Therewith, the claim follows immediately. Let us start with the derivation of \eqref{eqn:lowerBound:incremets:contPrart}. Therefore, let $L_{\mu_{cont}} > 0$ be the Lipschitz-constant of $\mu_{cont}$. Then an application of the H\"older-inequality and exploiting the Lipschitz-continuity of $\mu_{cont}$ yields
\begin{flalign*}
\EE & |\int_{0}^{t_i - t_{i-1}} (\mu_{cont}(X_{t_{i-1}} + V_s) - \mu_{cont}(X_{t_{i-1}} + \widetilde{V}_s)) \: ds |^2 
\\&\leq (L_{\mu_{cont}})^2 \cdot (t_i - t_{i-1}) \cdot \int_{0}^{t_i - t_{i-1}} \EE |V_s - \widetilde{V}_s|^2 \: ds
\leq (L_{\mu_{cont}})^2 \cdot (t_i - t_{i-1})^{3}.
\end{flalign*}
Due to $t_i - t_{i-1} \leq \frac{2}{n}$, the existence of $C_1$ in \eqref{eqn:lowerBound:incremets:contPrart} is shown. We will proceed with the bound in \eqref{eqn:lowerBound:incremets:discPrart} and we will use similar arguments as in the proof of Lemma 15 in~\cite{MGY21}. Let $i*\in \{1,\dots, k\}$ with $\mu_{disc}(\xi_{i*} +) \neq \mu_{disc}(\xi_{i*} -)$ as in $(\mu 3)$. It holds
\begin{flalign*}
\int_{0}^{t_i - t_{i-1}} &(\mu_{disc}(X_{t_{i-1}} + V_s) - \mu_{disc}(X_{t_{i-1}} + \widetilde{V}_s)) \: ds 
\\&= \int_{0}^{t_i - t_{i-1}} (\mu_{disc}(X_{t_{i-1}} + \frac{s}{t_i - t_{i-1}}(W_{t_i} - W_{t_{i-1}}) + B_{t_{i-1} + s}) 
\\& \quad \quad\quad\quad\quad\quad- \mu_{disc}(X_{t_{i-1}} + \frac{s}{t_i - t_{i-1}}(W_{t_i} - W_{t_{i-1}}) + \widetilde{B}_{t_{i-1} + s})) \: ds.
\end{flalign*}
Since $(X_{t_{i-1}} , \widetilde{X}_{t_{i-1}})$ and $((W_s - W_{t_{i-1}})_{s \in [t_{i-1}, t_i]}, (\widetilde{W}_s - \widetilde{W}_{t_{i-1}})_{s \in [t_{i-1}, t_i]})$ are independent due to Lemma \ref{lemma:representationOfXWithF}, we obtain the independence of
\\\centerline{$X_{t_{i-1}}, W_{t_i} - W_{t_{i-1}}, (B_{t_{i-1} + s})_{s \in [0, t_i - t_{i-1}]}, (\widetilde{B}_{t_{i-1} + s})_{s \in [0, t_i - t_{i-1}]}$.}
Hence, due to Lemma \ref{lemma:lowerBoundForSquaredIntegral} there exist constants $c_3, c_4, C_3 > 0$ such that 
\begin{flalign*}
\EE & |\int_{t_{i-1}}^{t_i} (\mu_{disc}(X_{t_{i-1}} + W_s - W_{t_{i-1}}) - \mu_{disc}(X_{t_{i-1}} + \widetilde{W}_s - \widetilde{W}_{t_{i-1}})) \: ds |^2 
\\&\geq c_3 (t_i - t_{i-1})^2\mathbb{P}(X_{t_{i-1}} \in [\xi_{i*}, \xi_{i*} + \sqrt{t_i - t_{i-1}}])- C_3 e^{- c_4 \slash (t_i - t_{i-1})}.
\end{flalign*}
Because of $t_i - t_{i-1} \leq \frac{2}{n}, t_i > \frac{1}{2}$ and with Lemma \ref{lemma:functionalsWelldefinedMuOfLinearGrowth} as well as Lemma \ref{lemma:boundProbability:SolutionIsInInterval}, the existence of the constants $c_1, c_2, C_2$ in \eqref{eqn:lowerBound:incremets:discPrart} follows.
\end{proof}

\begin{proof}[Proof of Theorem \ref{theorem:lowerBound:SDEs:discDrift}]
Let $L_{G_{\mu}} > 0$ be the Lipschitz-constant of $G_{\mu}$ which exists due to \mbox{Lemma \ref{lemma:basicPropertiesOfG}}. First of all, we note that it holds
\begin{flalign}
\EE|X_1 - \widetilde{X}_1|^2 \geq (L_{G_{\mu}})^{-2} \EE|G_{\mu}(X_1) - G_{\mu}(\widetilde{X}_1)|^2. \label{eqn:lowerBoundFinalTimepointWithTransform}
\end{flalign}
Now let us define for $i \in \{1, \dots, n\}$
\begin{flalign*}
\Delta_i := \EE|G_{\mu}(X_{t_i}) - G_{\mu}(\widetilde{X}_{t_i})|^2.
\end{flalign*}
Let $i \in \{1, \dots, n\}$ with $t_i > \frac{1}{2}$. Then it holds 
\begin{flalign}
\begin{aligned}
\Delta_i &= \EE \big[|G_{\mu}(X_{t_{i-1}}) - G_{\mu}(\widetilde{X}_{t_{i-1}}) + ((G_{\mu}(X_{t_i}) - G_{\mu}(X_{t_{i-1}})) - (G_{\mu}(\widetilde{X}_{t_i}) - G_{\mu}(\widetilde{X}_{t_{i-1}})))|^2 \big] 
\\&= \Delta_{i-1} + 2m_i + d_i 
\end{aligned}\label{eqn:rewriteDeltaWithMandD}
\end{flalign}
with
\begin{flalign*}
m_i := \EE\big[(G_{\mu}(X_{t_{i-1}}) - G_{\mu}(\widetilde{X}_{t_{i-1}})) \cdot ((G_{\mu}(X_{t_i}) - G_{\mu}(X_{t_{i-1}})) - (G_{\mu}(\widetilde{X}_{t_i}) - G_{\mu}(\widetilde{X}_{t_{i-1}})))\big]
\end{flalign*}
and
\begin{flalign*} 
d_i := \EE|(G_{\mu}(X_{t_i}) - G_{\mu}(X_{t_{i-1}})) - (G_{\mu}(\widetilde{X}_{t_i}) - G_{\mu}(\widetilde{X}_{t_{i-1}}))|^2.
\end{flalign*}
Due to Lemma \ref{lemma:boundForMixedTerms}, there exists a constant $C_1 > 0$ with
\begin{flalign}
|m_i| \leq \frac{C_1}{n} \Delta_{i-1}. \label{eqn:lowerBound:mi}
\end{flalign}
A combination of Lemma \ref{lemma:lowerBound:increments}, Lemma \ref{lemma:lowerBound:problemReductionToDistanceOfTildeProcesses}, Lemma \ref{lemma:lowerBound:dropRecursionIntegralOfSDE} and Lemma \ref{lemma:lowerBound:squaredIntegral:incremetsSimplified} yields the existence of constants $c_1, C_2, C_3 > 0$ with 
\begin{flalign}
\begin{aligned}
d_i &\geq c_1 (t_i - t_{i-1})^{5 \slash 2} - \frac{C_2}{n} \Delta_{i-1} - \frac{C_3}{n^{5 \slash 2 + 1 \slash 16}}.
\end{aligned}\label{eqn:lowerBound:di}
\end{flalign}
With \eqref{eqn:rewriteDeltaWithMandD}, \eqref{eqn:lowerBound:mi} and \eqref{eqn:lowerBound:di} we therefore obtain that there exists a constant $C_4> 0$ with
\begin{flalign*}
\Delta_i \geq (1- \frac{C_4}{n}) \Delta_{i-1} +  c_1 (t_i - t_{i-1})^{5 \slash 2} - \frac{C_3}{n^{5 \slash 2 + 1 \slash 16}}.
\end{flalign*}
Assume now $0 \leq \frac{C_4}{n} < 1$. Using the above estimate for $\Delta_i$ with $\Delta_{i-1}$ it follows inductively that for $j \in \{1, \dots, n\}$ with $t_j \geq \frac{1}{2}$ it holds
\begin{flalign*}
\Delta_n \geq (1- \frac{C_4}{n})^{n-j} \Delta_j + c_1 \sum_{i=j+1}^n(1- \frac{C_4}{n})^{n-i} (t_i - t_{i-1})^{5 \slash 2} - (n-j)\cdot \frac{C_3}{n^{5 \slash 2 + 1 \slash 16}}.
\end{flalign*}
Now take $r \in \{1, \dots, n\}$ with $t_r = \frac{1}{2}$. Then it holds
\begin{flalign*}
\EE|G_{\mu}(X_1) - G_{\mu}(\widetilde{X}_1)|^2 = \Delta_n \geq c_1(1- \frac{C_4}{n})^{n} \sum_{i=r+1}^n (t_i - t_{i-1})^{5 \slash 2} - n\cdot \frac{C_3}{n^{5 \slash 2 + 1 \slash 16}}.
\end{flalign*}
Since we know that $\lim_{n \rightarrow \infty} (1- \frac{C_4}{n})^{n} = e^{-C_4}$, we obtain for sufficiently large $n \in \mathbb{N}$
\begin{flalign}
\EE|G_{\mu}(X_1) - G_{\mu}(\widetilde{X}_1)|^2 = \Delta_n \geq \frac{c_1 e^{-C_4}}{2} \sum_{i=r+1}^n (t_i - t_{i-1})^{5 \slash 2} -  \frac{C_3}{n^{3 \slash 2 + 1 \slash 16}}. \label{eqn:lowerBoundWithSumOfIncremtensTi}
\end{flalign}
In the remaining part of the proof we will use ideas from the derivation of Lemma 16 in~\cite{MGY21}. An application of the H\"older-inequality with $p = \frac{5}{3}$ and $q= \frac{5}{2}$ yields 
\begin{flalign*}
\frac{1}{2} = \sum_{i=r+1}^{n}(t_i - t_{i-1}) \leq n^{3 \slash 5} \cdot \big( \sum_{i=r+1}^n (t_i - t_{i-1})^{5 \slash 2}\big)^{2 \slash 5}
\end{flalign*}
which can be rewritten as
\begin{flalign*}
\sum_{i=r+1}^n(t_i - t_{i-1})^{5 \slash 2} \geq \frac{1}{2^{5 \slash 2}n^{3 \slash 2}}.
\end{flalign*}
Plugging this into \eqref{eqn:lowerBoundWithSumOfIncremtensTi} shows that for sufficiently large $n \in \mathbb{N}$ it holds 
\begin{flalign*}
\EE|G_{\mu}(X_1) - G_{\mu}(\widetilde{X}_1)|^2 = \Delta_n \geq \frac{c_1 e^{-C_4}}{2^{7 \slash 2}} n^{- 3 \slash 2}   -  \frac{C_3}{n^{3 \slash 2 + 1 \slash 16}}.
\end{flalign*}
In view of \eqref{eqn:lowerBoundFinalTimepointWithTransform} we see that there exists a constant $c_2 > 0$ such that for all sufficiently large $n \in \mathbb{N}$ it holds
\begin{flalign*}
\EE |X_1 - \widetilde{X}_1|^2 \geq \frac{c_2}{n^{3 \slash 2}}.
\end{flalign*}
Now for $Z:= X_1 - \widetilde{X}_1$ an application of the H\"older-inequality shows that for all sufficiently large $n \in \mathbb{N}$ it holds
\begin{flalign*}
\frac{c_2}{n^{3 \slash 2}} \leq \EE |X_1 - \widetilde{X}_1|^2 \leq \big( \EE |X_1 - \widetilde{X}_1|\big)^{2 \slash 3} \cdot \big( \EE |X_1 - \widetilde{X}_1|^4\big)^{1 \slash 3}.
\end{flalign*}
Applying Lemma \ref{lemma:boundLPDistanceOfSolutionsAtInterpolation} with $p=4$ consequently yields the existence of a constant $C_5 > 0$ such that for all sufficiently large $n \in \mathbb{N}$ it holds
\begin{flalign*}
 \frac{c_2}{n^{3 \slash 2}} \leq \EE |X_1 - \widetilde{X}_1|^2 \leq \big( \EE |X_1 - \widetilde{X}_1|\big)^{2 \slash 3} \cdot C_5n^{-1}.
\end{flalign*}
Because of  
\begin{flalign*}
\inf_{\substack{t_1, \dots, t_m \in [0,1] \\g \colon \mathbb{R}^m \rightarrow \mathbb{R} \: measurable}} \EE |X_1 - g(W_{t_1}, \dots, W_{t_m})|  \geq \inf_{\substack{t_1, \dots, t_{m+1} \in [0,1] \\g \colon \mathbb{R}^{m+1} \rightarrow \mathbb{R} \: measurable}} \EE |X_1 - g(W_{t_1}, \dots, W_{t_{m+1}})| 
\end{flalign*}
for $m \in \mathbb{N}$ and Lemma \ref{lemma:lowerBound:infBoundedByTildeDistance} the claim follows.
\end{proof}

\subsection{Proof of Corollary \ref{cor:GeneralLowerBound}}\label{subsection:proofOfCorollary}
In this subsection we prove Corollary \ref{cor:GeneralLowerBound} which is basically an application of Lamperti's theorem.

\begin{proof}[Proof of Corollary \ref{cor:GeneralLowerBound}]
	We use the transformation 
	\begin{flalign*}
		\varphi: \mathbb{R} \rightarrow \mathbb{R}, \quad x \mapsto \int_0^x \frac{1}{\sigma(u)} \: du
	\end{flalign*}
	to apply Theorem \ref{theorem:lowerBound:SDEs:discDrift} later. Note that $\varphi$ is well-defined, bijective and Lipschitz-continuous because of $\inf_{x \in \mathbb{R}} |\sigma(x)| > 0$. Denoting the Lipschitz-constant of $\varphi$ by $L_\varphi$ we get
	\begin{flalign*}
		\inf_{\substack{t_1, \dots, t_n \in [0,1] \\g \colon \mathbb{R}^n \rightarrow \mathbb{R} \: measurable}} &\EE |X_1 - g(W_{t_1}, \dots, W_{t_n})|
		\geq L_\varphi^{-1} \inf_{\substack{t_1, \dots, t_n \in [0,1] \\g \colon \mathbb{R}^n \rightarrow \mathbb{R} \: measurable}} \EE |\varphi(X_1) - g(W_{t_1}, \dots, W_{t_n})|.
	\end{flalign*}
	Applying now the It\^{o}-formula yields that $Z= \varphi(X)$ is a strong solution of the SDE 
	\begin{flalign*}
		dZ_t &= \mu^\varphi(Z_t) \: dt + dW_t, \quad t\in [0,1],
		\\ Z_0 &= \varphi(x_0)
	\end{flalign*}
	with $\mu^\varphi = (\frac{\mu}{\sigma} - \frac{\sigma^\prime}{2}) \circ \varphi^{-1}$. 
	Using the assumptions of this corollary and the fact that the product of two bounded Lipschitz-continuous functions is again Lipschitz-continuous, one can show that $\mu^\varphi$ also satisfies $(\mu 1)$ and $(\mu 2)$. Moreover, $\mu^\varphi$ has the property $(\mu 3)$ because $(\mu 3)$ holds for $\mu$ and because of $\sigma \in C^3_b(\mathbb{R})$. An application of \mbox{Theorem \ref{theorem:lowerBound:SDEs:discDrift}} yields the result.
\end{proof}

\section*{Acknowledgement}
I want to express my gratitude to Thomas M\"uller-Gronbach and Larisa Yaroslavtseva for their encouragement and useful critiques of this article. I would also like to thank the referees for their careful reading and their suggestions.

\bibliographystyle{acm}
\bibliography{bibfile}

\end{document}